\documentclass{amsart}
\usepackage{amssymb}
\usepackage{amscd}
\usepackage{verbatim}
\usepackage{epsfig}
\usepackage[colorlinks=true,linkcolor=red,citecolor=blue,breaklinks=true]{hyperref}

\begin{document}
\newcommand\Mand{\ \text{and}\ }
\newcommand\Mor{\ \text{or}\ }
\newcommand\Mfor{\ \text{for}\ }
\newcommand\Real{\mathbb{R}}
\newcommand\RR{\mathbb{R}}
\newcommand\im{\operatorname{Im}}
\newcommand\re{\operatorname{Re}}
\newcommand\sign{\operatorname{sign}}
\newcommand\sphere{\mathbb{S}}
\newcommand\BB{\mathbb{B}}
\newcommand\HH{\mathbb{H}}
\newcommand\dS{\mathrm{dS}}
\newcommand\ZZ{\mathbb{Z}}
\newcommand\NN{\mathbb{N}}
\newcommand\codim{\operatorname{codim}}
\newcommand\Sym{\operatorname{Sym}}
\newcommand\End{\operatorname{End}}
\newcommand\Span{\operatorname{span}}
\newcommand\Ran{\operatorname{Ran}}
\newcommand\ep{\epsilon}
\newcommand\Cinf{\cC^\infty}
\newcommand\dCinf{\dot \cC^\infty}
\newcommand\CI{\cC^\infty}
\newcommand\dCI{\dot \cC^\infty}
\newcommand\Cx{\mathbb{C}}
\newcommand\Nat{\mathbb{N}}
\newcommand\dist{\cC^{-\infty}}
\newcommand\ddist{\dot \cC^{-\infty}}
\newcommand\pa{\partial}
\newcommand\Card{\mathrm{Card}}
\renewcommand\Box{{\square}}
\newcommand\Ell{\mathrm{Ell}}
\newcommand\Char{\mathrm{Char}}
\newcommand\WF{\mathrm{WF}}
\newcommand\WFh{\mathrm{WF}_\semi}
\newcommand\WFb{\mathrm{WF}_\bl}
\newcommand\WFsc{\mathrm{WF}_\scl}
\newcommand\WFscb{\mathrm{WF}_{\scl,\bl}}
\newcommand\Vf{\mathcal{V}}
\newcommand\Vb{\mathcal{V}_\bl}
\newcommand\Vsc{\mathcal{V}_\scl}
\newcommand\Vscsus{\mathcal{V}_\scsus}
\newcommand\Vz{\mathcal{V}_0}
\newcommand\Hb{H_{\bl}}
\newcommand\bHb{\bar H_{\bl}}
\newcommand\dHb{\dot H_{\bl}}
\newcommand\Hbb{\tilde H_{\bl}}
\newcommand\Hsc{H_{\scl}}
\newcommand\Hscb{H_{\scbl}}
\newcommand\bHscb{\bar H_{\scbl}}
\newcommand\dHscb{\dot H_{\scbl}}
\newcommand\Hscsus{H_{\scsus}}
\newcommand\Ker{\mathrm{Ker}}
\newcommand\Range{\mathrm{Ran}}
\newcommand\Hom{\mathrm{Hom}}
\newcommand\Id{\mathrm{Id}}
\newcommand\sgn{\operatorname{sgn}}
\newcommand\ff{\mathrm{ff}}
\newcommand\tf{\mathrm{tf}}
\newcommand\esssupp{\operatorname{esssupp}}
\newcommand\supp{\operatorname{supp}}
\newcommand\vol{\mathrm{vol}}
\newcommand\Diff{\mathrm{Diff}}
\newcommand\Diffd{\mathrm{Diff}_{\dagger}}
\newcommand\Diffs{\mathrm{Diff}_{\sharp}}
\newcommand\Diffb{\mathrm{Diff}_\bl}
\newcommand\Diffsc{\mathrm{Diff}_\scl}
\newcommand\Diffscsus{\mathrm{Diff}_\scsus}
\newcommand\DiffbI{\mathrm{Diff}_{\bl,I}}
\newcommand\Diffbeven{\mathrm{Diff}_{\bl,\even}}
\newcommand\Diffz{\mathrm{Diff}_0}
\newcommand\Psih{\Psi_{\semi}}
\newcommand\Psihcl{\Psi_{\semi,\cl}}
\newcommand\Psisc{\Psi_\scl}
\newcommand\Psiscc{\Psi_\sccl}
\newcommand\Psiscb{\Psi_\scbl}
\newcommand\Psib{\Psi_\bl}
\newcommand\Psibc{\Psi_{\mathrm{bc}}}
\newcommand\Psibcdelta{\Psi_{\mathrm{bc},\delta}}
\newcommand\TbC{{}^{\bl,\Cx} T}
\newcommand\Tb{{}^{\bl} T}
\newcommand\Sb{{}^{\bl} S}
\newcommand\Tsc{{}^{\scl} T}
\newcommand\Tscsus{{}^{\scsus} T}
\newcommand\Ssc{{}^{\scl} S}
\newcommand\Sscsus{{}^{\scsus} S}
\newcommand\Lambdab{{}^{\bl} \Lambda}
\newcommand\zT{{}^{0} T}
\newcommand\Tz{{}^{0} T}
\newcommand\zS{{}^{0} S}
\newcommand\dom{\mathcal{D}}
\newcommand\cA{\mathcal{A}}
\newcommand\cB{\mathcal{B}}
\newcommand\cE{\mathcal{E}}
\newcommand\cG{\mathcal{G}}
\newcommand\cH{\mathcal{H}}
\newcommand\cU{\mathcal{U}}
\newcommand\cO{\mathcal{O}}
\newcommand\cF{\mathcal{F}}
\newcommand\cM{\mathcal{M}}
\newcommand\cQ{\mathcal{Q}}
\newcommand\cR{\mathcal{R}}
\newcommand\cI{\mathcal{I}}
\newcommand\cL{\mathcal{L}}
\newcommand\cK{\mathcal{K}}
\newcommand\cC{\mathcal{C}}
\newcommand\cX{\mathcal{X}}
\newcommand\cY{\mathcal{Y}}
\newcommand\cXsus{\mathcal{X}_\sus}
\newcommand\cYsus{\mathcal{Y}_\sus}
\newcommand\cP{\mathcal{P}}
\newcommand\cS{\mathcal{S}}
\newcommand\cZ{\mathcal{Z}}
\newcommand\cW{\mathcal{W}}
\newcommand\Ptil{\tilde P}
\newcommand\ptil{\tilde p}
\newcommand\chit{\tilde \chi}
\newcommand\yt{\tilde y}
\newcommand\zetat{\tilde \zeta}
\newcommand\xit{\tilde \xi}
\newcommand\taut{{\tilde \tau}}
\newcommand\phit{{\tilde \phi}}
\newcommand\mut{{\tilde \mu}}
\newcommand\taubsemi{\tau_{\bl,\hbar}}
\newcommand\lambdasemi{\lambda_{\hbar}}
\newcommand\sigmat{{\tilde \sigma}}
\newcommand\sigmah{\hat\sigma}
\newcommand\zetah{\hat\zeta}
\newcommand\etah{\hat\eta}
\newcommand\taub{\tau_{\bl}}
\newcommand\mub{\mu_{\bl}}
\newcommand\taubh{\hat\tau_{\bl}}
\newcommand\mubh{\hat\mu_{\bl}}
\newcommand\nuh{\hat\nu}
\newcommand\loc{\mathrm{loc}}
\newcommand\compl{\mathrm{comp}}
\newcommand\reg{\mathrm{reg}}
\newcommand\GBB{\textsf{GBB}}
\newcommand\GBBsp{\textsf{GBB}\ }
\newcommand\bl{{\mathrm b}}
\newcommand\scl{{\mathrm{sc}}}
\newcommand\scbl{{\mathrm{sc,b}}}
\newcommand\sccl{{\mathrm{scc}}}
\newcommand\scsus{{\mathrm{sc-sus}}}
\newcommand\sus{{\mathrm{sus}}}
\newcommand{\sH}{\mathsf{H}}
\newcommand{\cte}{\digamma}
\newcommand\cl{\mathrm{cl}}
\newcommand\hsf{\mathcal{S}}
\newcommand\Div{\operatorname{div}}
\newcommand\hilbert{\mathfrak{X}}
\newcommand\smooth{\mathcal{J}}
\newcommand\decay{\ell}
\newcommand\symb{j}

\newcommand\Hh{H_{\semi}}

\newcommand\bM{\bar M}
\newcommand\Xext{X_{-\delta_0}}

\newcommand\xib{{\underline{\xi}}}
\newcommand\etab{{\underline{\eta}}}
\newcommand\zetab{{\underline{\zeta}}}

\newcommand\xibh{{\underline{\hat \xi}}}
\newcommand\etabh{{\underline{\hat \eta}}}
\newcommand\zetabh{{\underline{\hat \zeta}}}

\newcommand\zn{z}
\newcommand\sigman{\sigma}
\newcommand\psit{\tilde\psi}
\newcommand\rhot{{\tilde\rho}}

\newcommand\hM{\hat M}

\newcommand\Op{\operatorname{Op}}
\newcommand\Oph{\operatorname{Op_{\semi}}}

\newcommand\innr{{\mathrm{inner}}}
\newcommand\outr{{\mathrm{outer}}}
\newcommand\full{{\mathrm{full}}}
\newcommand\semi{\hbar}

\newcommand\Feynman{\mathrm{Feynman}}
\newcommand\future{\mathrm{future}}
\newcommand\past{\mathrm{past}}

\newcommand\elliptic{\mathrm{ell}}
\newcommand\diffordgen{k}
\newcommand\difford{2}
\newcommand\diffordm{1}
\newcommand\diffordmpar{1}
\newcommand\even{\mathrm{even}}
\newcommand\dimn{n}
\newcommand\dimnpar{n}
\newcommand\dimnm{n-1}
\newcommand\dimnp{n+1}
\newcommand\dimnppar{(n+1)}
\newcommand\dimnppp{n+3}
\newcommand\dimnppppar{n+3}

\newcommand\sob{s}

\newcommand\res{\mathrm{res}}
\newcommand\WFbr{\mathrm{WF}_{\bl,\res}}
\newcommand\Tbr{{}^{\bl,\res} T}
\newcommand\Sbr{{}^{\bl,\res} S}
\newcommand\Psibr{\Psi_{\bl,\res}}
\newcommand\Psibcr{\Psi_{\mathrm{bc},\res}}
\newcommand\Psiscbr{\Psi_{\scl,\bl,\res}}
\newcommand\tausc{\tau_{\res}}
\newcommand\musc{\mu_{\res}}
\newcommand\tausch{\hat\tau_{\res}}
\newcommand\musch{\hat\mu_{\res}}
\newcommand\Hscbr{H_{\scl,\bl,\res}}

\newcommand\nuc{\nu^c}

\newtheorem{lemma}{Lemma}[section]
\newtheorem{prop}[lemma]{Proposition}
\newtheorem{thm}[lemma]{Theorem}
\newtheorem{cor}[lemma]{Corollary}
\newtheorem{result}[lemma]{Result}
\newtheorem*{thm*}{Theorem}
\newtheorem*{prop*}{Proposition}
\newtheorem*{cor*}{Corollary}
\newtheorem*{conj*}{Conjecture}
\numberwithin{equation}{section}
\theoremstyle{remark}
\newtheorem{rem}[lemma]{Remark}
\newtheorem*{rem*}{Remark}
\theoremstyle{definition}
\newtheorem{Def}[lemma]{Definition}
\newtheorem*{Def*}{Definition}

\newcommand{\mar}[1]{{\marginpar{\sffamily{\scriptsize #1}}}}
\newcommand\av[1]{\mar{AV:#1}}

\renewcommand{\theenumi}{\roman{enumi}}
\renewcommand{\labelenumi}{(\theenumi)}

\title[Resolvent near zero energy, a Lagrangian approach]{Resolvent
  near zero energy on Riemannian scattering (asymptotically conic)
  spaces, a Lagrangian approach}
\author[Andras Vasy]{Andr\'as Vasy}
\address{Department of Mathematics, Stanford University, CA 94305-2125, USA}

\email{andras@math.stanford.edu}


\thanks{The author gratefully
  acknowledges partial support from the NSF under grant number
  DMS-1664683 and from a Simons Fellowship.}
\date{\today. Original version: May 29, 2019}

\begin{abstract}
We use a Lagrangian regularity perspective to discuss resolvent estimates near zero energy on Riemannian scattering, i.e.\
asymptotically conic, spaces, and their generalizations. In addition
to the Lagrangian perspective we introduce and use a resolved
pseudodifferential algebra to deal with zero energy degeneracies in a
robust manner.
\end{abstract}

\maketitle

\section{Introduction and outline}
The purpose of this paper is to describe the low energy behavior of the
resolvent on Riemannian scattering spaces $(X,g)$ of dimension $n\geq 3$ using a description that
focuses on the outgoing radial set by providing Lagrangian regularity
estimates. These spaces, introduced by Melrose \cite{RBMSpec}, are Riemannian manifolds which are asymptotic
to the `large end' of a cone; one example is asymptotically Euclidean spaces.
For $\sigma\neq 0$, including the
$|\sigma|\to\infty$ limit,
$\Delta_g-\sigma^2$ was studied in
\cite{Vasy:Limiting-absorption-lag}. We refer to the introduction of
that paper for a comparison of this approach, which uses a conjugation
to move the outgoing `spherical wave' asymptotics to the zero section
followed by second microlocalized at the zero section scattering
analysis (which means scattering-b analysis) on the one hand, and the more standard,
variable order space approach that in one way or another underlies a
number of the
proofs of the limiting absorption principle, including in
dynamical systems settings \cite{RBMSpec, Vasy:Minicourse,
  Faure-Sjostrand:Upper}, on the other. The recent paper \cite{Vasy:Zero-energy}
also analyzed the $\sigma\to 0$ behavior from the variable order
perspective; though this paper also used second microlocal techniques,
the reason was different: the degeneration of the characteristic set
as $\sigma\to 0$. In this paper we provide an alternative treatment to
\cite{Vasy:Zero-energy} that matches
\cite{Vasy:Limiting-absorption-lag}. We recall here that another area
in which a Lagrangian regularity (though without second microlocalization)
investigation has proved fruitful recently is describing internal waves in fluids,
see \cite{Dyatlov-Zworski:Forced}.

The study of the $\sigma\to 0$ limit has a long history, going back to
the work of Jensen and Kato \cite{Jensen-Kato:Spectral} in the
Euclidean setting. More recently
Guillarmou and Hassell analyzed this behavior in a series of
works
\cite{Guillarmou-Hassell:Resolvent-I,Guillarmou-Hassell:Resolvent-II}
via constructing a parametrix for the resolvent family; here we
proceed by {\em directly} obtaining Fredholm estimates. Other recent
works on the subject include those of  Bony and H\"afner
\cite{Bony-Haefner:Low}, Rodnianski and Tao
\cite{Rodnianski-Tao:Effective} and M\"uller and Strohmaier
\cite{Muller-Strohmaier:Hahn}; we refer to \cite{Vasy:Zero-energy} for
more details.

{\em This paper is intended as a companion paper to 
\cite{Vasy:Limiting-absorption-lag}, so the reader is advised to read 
that paper first for a more detailed introduction to the setting,
including the b- and scattering pseudodifferential operator algebras
and their relation to analysis on $\RR^n$, as 
well as for additional references.}

We recall that
second microlocal, spaces, see \cite[Section~5]{Vasy:Zero-energy} in this
scattering context, and see
\cite{Bony:Second,Vasy-Wunsch:Semiclassical} in different
contexts, play a role
in precise analysis at a Lagrangian, or more generally coisotropic,
submanifold. As mentioned above, these second microlocal
techniques played a role in \cite{Vasy:Zero-energy} due to the
degeneration of the principal symbol at zero energy, corresponding to
the quadratic vanishing of any dual metric function at the zero
section; the chosen Lagrangian is thus the zero section, really
understood as the zero section at infinity. In
a somewhat simpler way than in other cases, this second microlocalization at the zero
section is accomplished by simply using the
b-pseudodifferential operator algebra of Melrose \cite{Melrose:Atiyah}. In an informal way, this arises
by blowing up the zero section of the scattering cotangent bundle at
the boundary, though a more precise description (in that it makes
sense even at the level of
quantization, the spaces themselves are naturally diffeomorphic)  is
the reverse: blowing up the corner (fiber infinity over the boundary)
of the b-cotangent bundle. In \cite{Vasy:Zero-energy} this was used to
show a uniform version of the resolvent estimates down to zero energy
using variable differential order b-pseudodifferential
operators. Indeed, the
differential order of these, cf.\ the aforementioned blow-up of the
corner, corresponds to the scattering decay order away from the zero
section, thus this allows the uniform analysis of the problem to zero energy.
However, for this problem the decay order (of the b-ps.d.o.) is also crucial, for it corresponds to
the spaces on which the exact zero energy operator is Fredholm of
index zero, which,
with $\Hb$ denoting weighted b-Sobolev spaces relative to the
scattering (metric) $L^2$-density, are $\Hb^{\tilde r,l}\to\Hb^{\tilde
  r-2,l+2}$ with $|l+1|<\frac{n-2}{2}$, where $\tilde r$ is the
variable order (which is irrelevant at zero energy since the operator
is elliptic in the b-pseudodifferential algebra then). (The more
refined, fully 2-microlocal, spaces, corresponding to the blow-up of
the corner, have three orders: sc-differential,
sc-decay/b-differential and b-decay; using all of these is convenient,
as the operators are sc-differential-elliptic, so one can use easily that
there are no constraints on regularity in that sense; this modification is not crucial.)

Now, first for $\sigma\neq 0$ real, one can work in a second microlocal
space by simply conjugating the spectral family $P(\sigma)$ by
$e^{i\sigma/x}$ (this being the multiplier from the right),
with the point being that this conjugation acts as a canonical
transformation of the scattering cotangent bundle, moving the outgoing
radial set to the zero section. Then the second microlocal analysis
is simply a refinement of b-analysis. Indeed, note that this conjugation
moves $x^{-1}(x^2D_x+\sigma)$, resp.\ $x^{-1}(xD_{y_j})$, to
$x^{-1}(x^2D_x)=xD_x$, resp.\ $x^{-1}(xD_{y_j})=D_{y_j}$, so the
Lagrangian regularity becomes b-differential-regularity indeed. Notice
that the conjugate of the simplest model operator
$$
P(\sigma)=(x^2D_x)^2+i(n-1)x(x^2 D_x)+x^2\Delta_y-\sigma^2\in \Diffsc^2(X)\subset\Diffb^2(X)
$$
is then
\begin{equation*}\begin{aligned}
\hat
P(\sigma)&=e^{-i\sigma/x}P(\sigma)e^{i\sigma/x}=(x^2D_x-\sigma)^2+i(n-1)x(x^2
D_x-\sigma)+x^2\Delta_y-\sigma^2\\
&=(x^2D_x)-2\sigma(x^2D_x)+i(n-1)x(x^2
D_x)-i(n-1)x\sigma +x^2\Delta_y\in x\Diffb^2(X),
\end{aligned}\end{equation*}
which has one additional order of vanishing in this b-sense. (This is
basically the effect of the zero section of the sc-cotangent bundle
being now in the characteristic set.) Moreover, to leading order in
terms of the b-decay sense, i.e.\ modulo $x^2\Diffb^2(X)$, this is the
simple first order operator
$$
-2\sigma x\Big(xD_x+i\frac{n-1}{2}\Big).
$$
(In general, decay is
controlled by the normal operator of a b-differential operator,
which arises by setting $x=0$ in its coefficients after factoring out an
overall weight, and where one thinks of it as acting on functions on
$[0,\infty)_x\times\pa X$, of which $[0,\delta_0)_x\times\pa X$ is
identified with a neighborhood of $\pa X$ in $X$.)
This is non-degenerate for $\sigma\neq 0$ in that, on suitable spaces,
it has an invertible normal operator; of course, this is not an
elliptic operator, so some care is required. Notice that terms like
$(x^2D_x)^2$ and $\sigma x^2D_x$ have the same scattering decay order,
i.e.\ on the front face of the blown up b-corner they are equally
important. Thus, one does real principal type plus radial points
estimates in this case to conclude a Fredholm statement
\begin{equation}\label{eq:real-princ-b-orders}
\hat P(\sigma):\{u\in \Hb^{\tilde r,l}:\hat P(\sigma)u\in \Hb^{\tilde r,l+1}\}\to\Hb^{\tilde r,l+1},
\end{equation}
and here $l<-1/2$, $\tilde r+l>-1/2$ constant work. (Note that $\tilde
r+l$ is the scattering decay order away from the zero section.) We
refer to \cite{Vasy:Limiting-absorption-lag} for the proof.

Notice that, in terms of the limiting absorption principle, there are
two ways to implement this conjugation: one can conjugate either by
$e^{i\sigma/x}$, where $\sigma$ is now complex, or by
$e^{i\re\sigma/x}$. The former, which we follow, gives much stronger
spaces when $\sigma$ is not real with $\im\sigma>0$ (which is from where we
take the limit), as $e^{i\sigma/x}$ entails an
exponentially decaying weight $e^{-\im\sigma/x}$, so if the original
operator is applied to $u$, the conjugated operator is applied to
$e^{\im\sigma/x} u$.

One cannot expect an estimate that does
not already hold for the elliptic operator $\hat P(0)=P(0)$, and in
that case one has a Fredholm elliptic estimate in which the b-decay
order changes by $2$. Nonetheless, between spaces whose b-decay order
differs by 2 one might expect non-degenerate, uniform, estimate, as
was done in the unconjugated setting in \cite{Vasy:Zero-energy}. We might expect that this conjugated approach will be in some
ways more restrictive than the unconjugated one because two of the
three phenomena constraining orders (incoming and outgoing radial
points, and indicial roots, which are points of normal operator non-invertibility) are realized in the b-decay sense, namely
both the outgoing radial set and the indicial root phenomena take
place here.

This can be remedied by using a resolved version of the
b-pseudodifferential algebra which also allows for weights vanishing
precisely at $x=\sigma=0$, thus of the form $(x+|\sigma|)^\alpha$. We state the
estimates for resolved scattering-b Sobolev spaces $\Hscbr^{s,r,l}(X)$, or rather norms
(akin to the semiclassical spaces these are really a family of norms
on the same space), defined in
Section~\ref{sec:resolved} in \eqref{eq:Hscbr-def}. However, we already state that if $s=r-l$, this
is simply the standard b-Sobolev space $\Hb^{r-l,l}(X)$ with the
standard, $\sigma$-independent, norm, see \eqref{eq:Hscbr-Hb}; given $r,l$ satisfying the
hypotheses below, such a choice of $s$ is always acceptable for
$\Hscbr^{s,r,l}(X)$; then the $\Hscbr^{s-2,r+1,l+1}(X)$ norm on $\hat
P(\sigma)u$ can be strengthened to $\Hscbr^{s,r+1,l+1}=\Hb^{r-l,l+1}$ (thus the
estimate weakened) which is of the same form, but this is quite
lossy; we give below in \eqref{eq:main-b} a better (but still lossy) version.

We formulate the general theorem below in Section~\ref{sec:operator},
in Theorem~\ref{thm:main-gen}, where various additional notations are
introduced. However, here we state
our main theorem for the spectral family $P(\sigma)$ for an
asymptotically conic Laplacian $\Delta_g$:

\begin{thm}\label{thm:main}
Suppose that $|l'+1|<\frac{n-2}{2}$, and suppose that
$P(0):\Hb^{\infty,l'}\to\Hb^{\infty,l'+2}$ has trivial nullspace, an
assumption independent of $l'$ in this range. Suppose also that either
$r>-1/2$, $l<-1/2$, or $r<-1/2$, $l>-1/2$. Let
$$
\hat
P(\sigma)=e^{-i\sigma/x}P(\sigma)e^{i\sigma/x}.
$$

There exists
$\sigma_0>0$ such
that
$$
\hat P(\sigma):\{u\in\Hscbr^{s,r,l}:\ \hat P(\sigma)u\in \Hscbr^{s-2,r+1,l+1}\}\to\Hscbr^{s-2,r+1,l+1}
$$
is invertible for $0<|\sigma|\leq\sigma_0$, $\im\sigma\geq 0$, with this inverse being the $\pm i0$ resolvent of
$P(\sigma)$ corresponding to $\pm\re\sigma>0$, and we have the
estimate
\begin{equation*}\begin{aligned}
&\|(x+|\sigma|)^{\alpha}u\|_{\Hscbr^{s,r,l}}\leq C\|(x+|\sigma|)^{\alpha-1}\hat
P(\sigma)u\|_{\Hscbr^{s-2,r+1,l+1}}
\end{aligned}\end{equation*}
for
$$
\alpha\in\Big(l+1-\frac{n-2}{2},l+1+\frac{n-2}{2}\Big).
$$
\end{thm}

\begin{rem}\label{rem:main-b}
We remark that the estimate implies the following estimate purely in terms
of b-Sobolev spaces. Suppose $s=r-l$, with $r,l$ as in the theorem. Then
\begin{equation}\begin{aligned}\label{eq:main-b}
&\|(x+|\sigma|)^{\alpha}u\|_{\Hb^{s,l}}\leq C\|(x+|\sigma|)^{\alpha}\hat
P(\sigma)u\|_{\Hb^{s-1,l+2}}.
\end{aligned}\end{equation}
We refer to \eqref{eq:main-b-proof} below for its proof, and to
\eqref{eq:main-scb-proof} for a further strengthened (non-resolved)
scattering-b, i.e.\ second microlocal, statement, which nonetheless is
still weaker than the main theorem.
\end{rem}

Note that in case we want $l<-1/2$ (so that the b-decay order is low,
but b-differentiability/sc-decay is high), one can always (for $n\geq
2$, e.g.\ $l$ close to $-1/2$) take
$\alpha=0$ for suitable $l$, while if we want $l>-1/2$ (so that the b-decay order is high,
but b-differentiability/sc-decay is low) we can only do this for
$n\geq 4$, otherwise the lower limit of the $\alpha$ interval is $>0$.

We remark that in spite of the earlier indication that two of the three
phenomena constraining orders take place at the same location, for the $l<-1/2$ case in 3 dimensions one may expect,
as we prove, an optimal 
result with $\alpha=0$ because the decay order of spherical waves (which is the 
threshold order for the outgoing Lagrangian) and that of the zero 
energy Green's function (which is the beginning of the interval 
allowed by the indicial roots) are the same: $x=r^{-1}$.

Note also the estimate can be rewritten as
\begin{equation*}\begin{aligned}
&\|(1+x/|\sigma|)^{\alpha}u\|_{\Hscbr^{s,r,l}}\leq C\|(1+x/|\sigma|)^{\alpha}(x+|\sigma|)^{-1}\hat
P(\sigma)u\|_{\Hscbr^{s-2,r+1,l+1}}.
\end{aligned}\end{equation*}

The structure of this paper is the following. In
Section~\ref{sec:operator} we describe the structure of the general
class of operators we are considering. In Section~\ref{sec:resolved}
we introduce the resolved pseudodifferential algebras that allow for
precise estimates down to $\sigma=0$. In Section~\ref{sec:symbolic} we
obtain symbolic estimates whose errors gain in the scattering differential and
decay orders. In Section~\ref{sec:normal} we remove those errors
using a normal operator estimate that is obtained by reducing to the
non-zero spectral parameter case analyzed in
\cite{Vasy:Limiting-absorption-lag} via a rescaling argument. Finally
in Section~\ref{sec:zero} we discuss what happens in a simple case
when $P(0)$ does have a non-trivial nullspace.

I am very grateful for numerous discussions with Peter Hintz, various
projects with whom have formed the basic motivation for this work, and
whose comments helped improve this manuscript.
I
also thank Dietrich H\"afner and  Jared Wunsch for their interest in
this work which helped to push it towards completion.

\section{The operator}\label{sec:operator}
First recall the framework in which the $\sigma\to 0$ behavior was analyzed in the
unconjugated setting in \cite{Vasy:Zero-energy}.
To start with, we have a scattering metric $g\in S^0(\Tsc^*X\otimes_s\Tsc^*X)$
for which there is an actually conic metric
$g_0=x^{-4}\,dx^2+x^{-2}h$, $h$ a Riemannian metric on $\pa X$, to
which $g$ is asymptotic in the sense that
$g-g_0\in S^{-\tilde\delta}(X,\Tsc^*X\otimes_s\Tsc^*X)$, $\tilde\delta>0$. Then
\begin{equation*}\begin{aligned}
&P(\sigma)=P(0)+\sigma Q-\sigma^2,\\
&\qquad P(0)\in S^{-2}\Diffb^2(X),\ Q\in
S^{-2-\tilde\delta}\Diffb^1(X),\ P(0)-\Delta_g\in S^{-2-\tilde\delta}\Diffb^2(X)
\end{aligned}\end{equation*}
thus also
$$
P(0)-\Delta_{g_0}\in S^{-2-\tilde\delta}\Diffb^2(X).
$$
In the present paper we obtain more precise information than in
\cite{Vasy:Zero-energy}, but under assumptions which are stronger on
the highest (second) order terms, though more relaxed on the lower order terms,
namely $Q$, as well as $\sigma^2$ terms (though we recall that in
\cite{Vasy:Zero-energy} $Q$ was allowed to smoothly depend on $\sigma$). Thus, we
take $\tilde\delta=1$, and we impose the existence of leading terms, so
$$
g-g_0\in x\CI (X,\Tsc^*X\otimes_s\Tsc^*X)
+S^{-1-\delta}(X,\Tsc^*X\otimes_s\Tsc^*X),\qquad\delta>0,
$$
for the metric. We allow a more general form for the operator in terms
of the coefficient of $\sigma^2$:
\begin{equation*}\begin{aligned}
&P(\sigma)=P(0)+\sigma Q-\sigma^2(1-R),
\end{aligned}\end{equation*}
and we take
\begin{equation}\begin{aligned}\label{eq:PQR-b-def}
&P(0)-\Delta_g\in x^2\Diffb^1(X)+S^{-2-\delta}\Diffb^1(X)\subset x\Diffsc^1(X)+S^{-1-\delta}\Diffsc^1(X),\\
&Q\in x\Diffsc^1(X)+S^{-1-\delta}\Diffsc^1(X)\\
&\qquad\qquad=x\CI(X)+S^{-1-\delta}(X)+x^2\Diffb^1(X)+S^{-2-\delta}\Diffb^1(X),\\
&R\in x\CI(X)+S^{-1-\delta}(X),
\end{aligned}\end{equation}
thus also
$$
P(0)-\Delta_{g_0}\in x^2\Diffb^1(X)+S^{-2-\delta}\Diffb^2(X).
$$
Note that the membership of $P(0)-\Delta_g$ in $x\Diffsc^1(X)+S^{-1-\delta}\Diffsc^1(X)$ is the condition used in
\cite{Vasy:Limiting-absorption-lag} and indeed any $P(\sigma)$
satisfying the requirements here satisfies those of \cite{Vasy:Limiting-absorption-lag}; the requirement here is stronger as
it rules out terms in $x\CI(X)+S^{-1-\delta}(X)$, such as Coulomb type
potentials. However, such $x\CI(X)+S^{-1-\delta}(X)$ terms are
allowed in $Q$ and $R$ due to the prefactor $\sigma$ or $\sigma^2$
present in front of them.

In \cite{Vasy:Limiting-absorption-lag} there was no need for stronger
assumptions for the skew-adjoint parts of operators, essentially
because they are subprincipal in terms of sc-decay, so while they
affect the statements (via shifting threshold regularity values), they
can be handled. Here, for our more delicate problem,
we also demand the stronger statements that
\begin{equation}\label{eq:Im-P}
\frac{1}{2i}(P(0)-P(0)^*)-\beta_I x \Big(x^2
D_x+ix\frac{n-2}{2}\Big)-\beta'_I x^2\in
S^{-2-\delta}\Diffb^1(X)\subset S^{-1-\delta}\Diffsc^1(X)
\end{equation}
for some $\beta_I,\beta'_I\in\CI(X)$ (which can simply be thought of as
functions on $\pa X$, as the $x\CI(X)$ terms can be absorbed into the
right hand side, and the $ix\frac{n-2}{2}$ term is included in $\beta_I$
as opposed to $\beta'_I$ since $x (x^2 D_x+ix\frac{n-2}{2})$ is
formally self-adjoint modulo terms that can be incorporated into the
right hand side), and
\begin{equation}\label{eq:Im-Q}
\frac{1}{2i}(Q-Q^*)-\gamma_I x\in S^{-1-\delta}(X)+S^{-2-\delta}\Diffb^1(X)=S^{-1-\delta}\Diffsc^1(X),
\end{equation}
for suitable $\gamma_I\in\CI(X)$, and
\begin{equation}\label{eq:Im-R}
\frac{1}{2i}(R-R^*)\in S^{-1-\delta}(X).
\end{equation}

As, $P(\sigma)\in\Psib^{2,0}$ only, and in the usual
sense the normal operator in $\Psib^{2,0}$ is simply $-\sigma^2$ as
$P(\sigma)+\sigma^2\in\Psib^{2,-2}$. Thus, in \cite{Vasy:Zero-energy} we instead considered the
`effective normal operator', quotienting the operator by
$S^{-2-\tilde\delta}\Diffb^2(X)$, which under the assumptions of
\cite{Vasy:Zero-energy} yields
$$
\tilde N(P(\sigma))=N(P(0))-\sigma^2=\Delta_{g_0}-\sigma^2,
$$
so
$$
P(\sigma)-\tilde N(P(\sigma))\in S^{-2-\tilde\delta}\Diffb^2(X);
$$
this difference was irrelevant for the analysis of b-decay. Here, due to our weaker assumptions on $Q$ as well as $P(0)$, the aforementioned extended normal
operator would in fact also include the leading order terms of $Q$ as
they are in $x^2\Diffb^1(X)$, and it would also include the
$x^2\Diffb^1(X)$ terms from $P(0)$ as well as more than just the
leading order terms from $R$.

From the Lagrangian perspective we consider a conjugated version of $P(\sigma)$. Thus,
let
$$
\hat P(\sigma)=e^{-i\sigma/x}P(\sigma)e^{i\sigma/x}.
$$
Since conjugation by $e^{i\sigma/x}$ is well-behaved in the
scattering, but not in the b-sense, it is actually advantageous to
first perform the conjugation in the scattering setting, and then
convert the result to a b-form. The principal symbol of $\hat P$ in
the scattering decay sense is
simply a translated version, by $d(\sigma/x)$, of that of $P$ (this
corresponds to $e^{-i\sigma/x}(x^2D_x)e^{i\sigma/x}=x^2D_x-\sigma$), which
is $\tau^2+\mu^2-\sigma^2$, thus it is
$\tau^2+\mu^2-2\sigma\tau$. However, we need more precise information,
thus we perform the computation explicitly.

\begin{prop}\label{prop:full-hat-P-sigma}
We have
\begin{equation}\label{eq:full-hat-P-sigma-separated}
\hat P(\sigma)=\hat P(0)+\sigma\hat Q+\sigma^2\hat R-2\sigma \Big(x^2D_x+i\frac{n-1}{2}x+\frac{\hat\beta-\hat\gamma}{2}x\Big)
\end{equation}
with $\hat\beta,\hat\gamma\in \CI(X)+S^{-\delta}(X)$, $\im\hat\beta|_{\pa
  X}=\beta_I$, $\im\hat\gamma|_{\pa X}=\gamma_I$,
\begin{equation*}\begin{aligned}
\hat P(0)&=P(0)\in x^2\Diffb^2(X)+S^{-2-\delta}\Diffb^2(X),\\
\hat Q&\in x^2\Diffb^1(X)+S^{-2-\delta}\Diffb^1(X),\\
\hat R&\in x\CI(X)+S^{-1-\delta}(X);
\end{aligned}\end{equation*}
and $\im \hat Q\in S^{-1-\delta}(X)+S^{-2-\delta}\Diffb^1(X)$, $\im\hat R\in S^{-1-\delta}(X)$.
\end{prop}

\begin{rem}
Notice that $\hat P(\sigma)\in x\Diffb^2(X)$ (modulo the faster
decaying $S^{-2-\delta}\Diffb^2(X)$), unlike $P(\sigma)$ which
is merely in $\Diffb^2(X)$ (modulo the faster
decaying $S^{-1-\delta}\Diffb^2(X)$) due to the $\sigma^2$ term; this one order
decay improvement plays a key role below.

Also, in \cite[Equation~(3.5)]{Vasy:Limiting-absorption-lag} one has
$\hat P(0)=P(0)-xa'$ with the notation there (so $xa'$ there is
$x^2a'$ here, see the notation in the proof below); under our present assumptions of it being $O(x^2)$, we
do not need to remove the $a'$ term from $P(0)$.
\end{rem}

\begin{proof}
In general we have
\begin{equation*}\begin{aligned}
P(0)=&(1+xa_{00})(x^2D_x)^2+\sum_j xa_{0j} ((x^2D_x)
(xD_{y_j})+(xD_{y_j}) (x^2D_x))\\
&+\sum_{i,j} a_{ij}
(xD_{y_i})(xD_{y_j})\\
&+(i(n-1)+\beta+ a_0)x(x^2D_x)+\sum_j xa_j (xD_{y_j})+x^2 a',
\end{aligned}\end{equation*}
and
\begin{equation*}\begin{aligned}
Q=b_0 x (x^2D_x)+\sum_j xb_j (xD_{y_j})+\gamma x+b'x
\end{aligned}\end{equation*}
with $a_{00},a_{0j},a_j,a',b_0,b_j\in \CI(X)+S^{-\delta}(X)$,
$a_0,b'\in S^{-\delta}(X)$, $a_{ij}\in
\CI(X)+S^{-1-\delta}(X)$,  $\im a_j,\im b_0,\im b_j\in
S^{-\delta}(X)$, $a'-\beta'\in S^{-\delta}(X)$, $\im\beta'=\beta'_I+\frac{n-2}{2}$, $\beta,\gamma\in\CI(X)$ (which can be
considered as functions on $\pa X$ due to the $a_0$ and $b'$ terms,
they are singled out rather than included in $a_0,b'$ due to their
role below) with $\im\beta=\beta_I$, $\im\gamma=\gamma_I$, and with
$b_0,b_j,b'$ smoothly depending on $\sigma$. Let
$$
\hat\beta=\beta+a_0,\ \hat\gamma=\gamma+b'.
$$
As $e^{-i\sigma/x}(x^2D_x)e^{i\sigma/x}=x^2D_x-\sigma$, this gives
\begin{equation*}\begin{aligned}
&e^{-i\sigma/x}P(0)e^{i\sigma/x}\\
=&(1+xa_{00})(x^2D_x-\sigma)^2+\sum_j xa_{0j} ((x^2D_x-\sigma)
(xD_{y_j})+(xD_{y_j}) (x^2D_x-\sigma))\\
&\qquad+\sum_{i,j} a_{ij}
(xD_{y_i})(xD_{y_j})+(i(n-1)+\beta+a_0)x(x^2D_x-\sigma)\\
&\qquad+\sum_j xa_j (xD_{y_j})+x^2 a',
\end{aligned}\end{equation*}
and
$$
e^{-i\sigma/x}Qe^{i\sigma/x}=b_0 x (x^2D_x-\sigma)+\sum_j x b_j
(xD_{y_j})+\gamma x+b'x.
$$
Combining the terms, including $R$, gives
\begin{equation}\label{eq:full-hat-P-sigma-separated-b}
\hat P(\sigma)=\hat P(0)+\sigma\hat Q+\sigma^2\hat R-2\sigma \Big(x^2D_x+i\frac{n-1}{2}x+\frac{\hat\beta-\hat\gamma}{2}x\Big)
\end{equation}
with
\begin{equation*}\begin{aligned}
\hat P(0)&=P(0)\in x^2\Diffb^2(X)+S^{-2-\delta}\Diffb^2(X),\\
\hat Q&=Q-2xa_{00}(x^2D_x)-2\sum_j xa_{0j}(xD_{y_j})-x\hat\gamma\\
&\qquad\qquad\qquad\qquad\in x^2\Diffb^1(X)+S^{-2-\delta}\Diffb^1(X),\\
\hat R&=R+xa_{00}-xb_0\in x\CI(X)+S^{-1-\delta}(X);
\end{aligned}\end{equation*}
note that $\im\hat R,\im\hat Q$ are also as stated.
\end{proof}

We also remark that the principal symbol of $\hat P(0)$ vanishes
 quadratically at the scattering zero section, $\tau=0$, $\mu=0$,
 $x=0$, hence the subprincipal symbol makes sense directly there (without
 taking into account contributions from the principal symbol, working
 with half-densities, etc.),
 and this in turn vanishes. It is convenient to summarize this,
 including positivity properties of $\hat P(0)$ here, as this 
will be helpful when considering non-real $\sigma$ below. Note that
this result already appears in \cite{Vasy:Limiting-absorption-lag};
the stronger assumptions in our case do not affect the statement.

\begin{lemma}[cf.\ Lemma~3.2 of \cite{Vasy:Limiting-absorption-lag}]
The operator $\hat P(0)$ is non-negative modulo terms that are either
sub-sub-principal or subprincipal but with vanishing contribution at
the scattering zero section, in the sense that it
has the form
\begin{equation}\label{eq:hat-P-0-nonnegative}
\hat P(0)=\sum_j T_j^*T_j+\sum_j T_j^* T'_j+\sum_j T^\dagger_j T_j+T''
\end{equation}
where $T_j\in x\Diffb^1(X)+S^{-2-\delta}\Diffb^1(X)$, $T'_j,T^\dagger_j\in
x\CI(X)+S^{-1-\delta}(X)$, $T''\in x^2\CI(X)+S^{-2-\delta}(X)$.
Moreover,
\begin{equation}\label{eq:hat-Q-module-form}
\hat Q=\sum_j T_j^* \tilde T'_j+\sum_j \tilde T^\dagger_j T_j+\tilde T''
\end{equation}
with $\tilde T'_j,\tilde T_j^\dagger\in
x\CI(X)+S^{-1-\delta}(X)$, $\tilde T''\in x^2\CI(X)+S^{-2-\delta}(X)$.
\end{lemma}

The standard normal operator
of $\hat P(\sigma)$, which arises by considering the operator
$x^{-1}\hat P(\sigma)$ and freezing the coefficients at
the boundary,
\begin{equation}\begin{aligned}\label{eq:actual-normal-op-tilde}
N(\hat P(\sigma))&=-2\sigma \Big(x^2D_x+i\frac{n-1}{2}x+\frac{\beta-\gamma}{2}x\Big)+\sigma^2
x \varpi\\
&=-2\sigma \Big(x^2D_x+i\frac{n-1}{2}x +\frac{\beta-\gamma}{2}x-\sigma\frac{\varpi}{2}x\Big),\\
&\qquad \varpi=(x^{-1}R)|_{\pa X}+a_{00}|_{\pa X}-b_0|_{\pa X},\
\beta=\hat\beta|_{\pa X},\ \gamma=\hat\gamma|_{\pa X},
\end{aligned}\end{equation}
degenerates at $\sigma=0$,
corresponding to $\hat P(\sigma)$ being in $x^2\Diffb^2(X)+\sigma
x\Diffb^1(X)$ (modulo faster decaying terms with symbolic coefficients), so the
definiteness of the operator at $\sigma=0$  still arises from $x^2\Diffb^2(X)$.
Hence,
we need to use an
effective normal operator even with
this approach, which thus again will not be dilation
invariant. However, we shall use a joint scaling in $(x,\sigma)$, in
which sense it is well behaved.

Before proceeding, we remark that the actual normal operator,
\eqref{eq:actual-normal-op-tilde}, is $x$ times the normal vector
field to the boundary plus a scalar, which, for $\sigma\neq 0$, corresponds to the
asymptotic behavior of the solutions of $\hat P(\sigma)v\in\dCI(X)$
being
$$
x^{(n-1-i (\beta-\gamma)+i\sigma\varpi)/2}\CI(\pa X),
$$
modulo faster decaying terms. This corresponds to the asymptotics
$$
e^{i\sigma/x}x^{(n-1-i(\beta-\gamma)+i\sigma\varpi)/2}\CI(\pa X)
$$
for solutions of $P(\sigma)u\in\dCI(X)$ for $\sigma\neq 0$. This
indicates that we can remove the contribution of $\varpi,\beta,\gamma$ to
leading decay order by conjugating the operator by
$x^{(-i(\beta-\gamma)+i\sigma\varpi)/2}$, but we do not do this
here. We also remark that $\varpi$ is real by our assumptions. Notice that in the case of Kerr spacetimes, if we factor out the
coefficient of $\pa_t^2$ from the operator, $a_{00}|_{\pa X}=-4m$ and
$(x^{-1}R)|_{\pa X}=0$, $b_0|_{\pa
  X}=0$,
so our
conjugating factor is
$e^{i\sigma (r+2m\log r)}$, which is asymptotically exactly
$e^{i\sigma r_*}$, $r_*$ the logarithmically modified radial function
(a version of the Regge-Wheeler radial tortoise coordinate function). If we do not factor this coefficient out, then
$a_{00}|_{\pa X}=-2m$, $\hat R-(-2mx)\in x^2\CI(X)$, and $b_0|_{\pa
  X}=0$, so we obtain the same conclusion.

For us, the key normal operator is the one
associated to the front face of the blow up of $x=\sigma=0$. In the
present case it captures $\hat P(\sigma)$ modulo
$$
x(x+\sigma)(x^\delta+\sigma)S^0\Diffb^2(X)
\subset\Psibr^{2,-1,-2-\delta,0}(X),
$$
with the latter class of pseudodifferential operators introduced in
the next section.
Note that this operator does not quite encapsulate the standard normal operator
since in the second, b-decay, order there is no gain over the a priori given
membership of $\hat P(\sigma)$, but we can use the smallness of
$\sigma$ (we are interested in the zero energy limit after all) to
deal with this, see the proof of Proposition~\ref{prop:b-res-imp-est}.

\begin{Def}
We define the {\em effective normal operator} $N_0(\hat P(\sigma))$ as
$\hat P(\sigma)$ modulo $x(x+\sigma)(x^\delta+\sigma)S^0\Diffb^2(X)$.
\end{Def}

The nice feature is that starting from sc-differential operators,
which we consider as b-operators with decaying coefficients, and
conjugating them by exponentials, we lose decay but gain a factor of
$\sigma$, so if it were not for the overall $x$ vanishing in the space
to be quotiented out, $x(x+\sigma)(x^\delta+\sigma)S^0\Diffb^2(X)$,
any term in $S^{-2-\delta}\Diffb^2(X)$ in the unconjugated operator
$P(\sigma)$ would automatically
give a trivial contribution to the effective normal operator. As is,
in $P(0)$ more structure is needed, which is the reason for giving the
requirements in \eqref{eq:PQR-b-def} in the stated form.

We then have
$$
N_0(\hat P(\sigma))=\Delta_{g_0}+\beta x^2\Big(xD_x+i\frac{n-2}{2}\Big)+x^2\beta'-2\sigma \Big(x^2D_x+i\frac{n-1}{2}x +\frac{\beta-\gamma}{2}x\Big)
$$
in the sense that
$$
N_0(\hat P(\sigma))-\hat P(\sigma)\in x(x+\sigma)(x^\delta+\sigma)\Diffb^2(X)\subset\Psibr^{2,-1,-2-\delta,0}(X).
$$
More precisely, we identify $X$ near $\pa X$ with $[0,x_0)_x\times \pa
X$, as usual for the standard normal operator, regard $N_0(\hat
P(\sigma))$ as an operator on the cone over $\pa X$,
$[0,\infty)_x\times\pa X$, and the requirement is that {\em evaluated on} $\sigma$-dependent families supported in $x<x_0$, with output
restricted to the same region, the difference of $N_0(\hat P(\sigma))$
and $\hat P(\sigma)$ has the desired form, i.e.\ is given by an
operator family with the indicated properties. A key point is that
$N_0(\hat P(\sigma))$ is dilation invariant jointly in $(x,\sigma)$,
which we shall use in Section~\ref{sec:normal}.

In addition to the extended normal operator, we also need to consider
the {\em standard} normal operator $N(\hat P(0))$ of $\hat P(0)$
as an operator in
$$
x^2\Diffb^2(X)+S^{-2-\delta}\Diffb^2(X),
$$
thus quotienting out by $S^{-2-\delta}\Diffb^2(X)$. Correspondingly,
we keep more information (for this term) than for $N(\hat
P(\sigma))$, since there the quotient is by $S^{-1-\delta}\Diffb^2(X)$
but on the other hand this is simply the extended normal operator
$N_0(\hat P(0))$ of $\hat P(0)$, namely it is
$$
\Delta_{g_0}+\beta x^2\Big(xD_x+i\frac{n-2}{2}\Big)+x^2\beta'\in x^2\Diffb^2(X),
$$
modulo $S^{-2-\delta}\Diffb^2(X)$.

For normal operator purposes it is convenient to work with $L^2_{\bl}$ instead of the metric
$L^2$-space, $L^2_{g_0}$; this is given by the density
$\frac{dx}{x}\,dh_0$, $h_0=h|_{\pa X}$ the metric on the cross
section of the asymptotic cone, so
$$
L^2_\bl=x^{-n/2}L^2_{g_0}.
$$ 
Let
$$
\Delta_{\bl}=x^{-(n+2)/2}\Delta_{g_0}x^{(n-2)/2}\in\Diffb^2(X);
$$
as $x^{-1}\Delta_{g_0} x^{-1}$ is symmetric with respect to the
$L^2_{g_0}$-inner product, $\Delta_{\bl}$ is symmetric with respect to the
$L^2_\bl$ inner product. Explicitly, with $\Delta_{\pa X}=\Delta_{h_0}$,
\begin{equation}\begin{aligned}\label{eq:b-Lap}
\Delta_{\bl}&=x^{n/2}D_x x^{-n+3}D_xx^{n/2-1}+\Delta_{\pa X}\\
&=\Big(D_x x+i\frac{n}{2}\Big) x^{-n/2+2}D_xx^{n/2-1}+\Delta_{\pa X}\\
&=\Big(D_x x+i\frac{n}{2}\Big) \Big(xD_x-i\frac{n-2}{2}\Big)+\Delta_{\pa X}\\
&=(xD_x)^2 +\Delta_{\pa X}+\Big(\frac{n-2}{2}\Big)^2;
\end{aligned}\end{equation}
notice that this is a positive definite operator on $L^2_\bl$ for
$n\geq 3$, since on the
Mellin transform side it is multiplication by a positive (operator valued) function.
The full conjugated and re-normalized operator (in that $x^2$ is
factored out) is
$$
x^{-(n+2)/2}\Big(\Delta_{g_0}+\beta x^2\Big(xD_x+i\frac{n-2}{2}\Big)+x^2\beta'\Big)
x^{(n-2)/2}=\Delta_{\bl}+\beta (xD_x)+\beta'.
$$
Mellin transforming in $x$ we obtain the elliptic family
$$
\taub^2+\beta\taub+\beta'+\Delta_{\pa X}+\Big(\frac{n-2}{2}\Big)^2
$$
of operators on $\pa X$, which is also elliptic in the large parameter
sense (in $\taub$, with $\im\taub$ bounded), invertible for large
$|\taub|$ with $\im\taub$ bounded, so the inverse is a meromorphic
family. Its poles are called the {\em indicial roots}. If $\beta,\beta'$ are constant scalars, this is invertible whenever
$$
-\Big(\frac{n-2}{2}\Big)^2-\taub^2-\beta\taub-\beta'
$$
is not an eigenvalue of $\Delta_{\pa X}$; if $\beta,\beta'$ are
constant non-scalar and have a
joint eigenspace decomposition ($\beta$ is assumed to be skew-adjoint
below!), then one can effectively replace them by the eigenvalues. Thus, the indicial roots are
of the form
$$
\frac{1}{2}\Big(-\beta\pm\sqrt{\beta^2-4\Big(\lambda+\Big(\frac{n-2}{2}\Big)^2+\beta'\Big)}\,\Big),
$$
with $\lambda$ an eigenvalue of $\Delta_{\pa X}$, which for $\beta=0$, $\beta'=0$ reduces to
$$
\pm i\sqrt{\lambda+\Big(\frac{n-2}{2}\Big)^2}.
$$
This means that for $\beta=0$, $\beta'=0$ one
has a `central interval' $(-\frac{n-2}{2}, \frac{n-2}{2})$ such that
if $\im\taub$ is in the interval, then the Mellin transformed normal
operator is invertible. This corresponds to invertibility of the
original operator on weighted $L^2_\bl$ spaces $x^{\ell}L^2_\bl$, where
$|\ell|<\frac{n-2}{2}$. This means that the unconjugated operator
$N(\hat P(0))$ is invertible from weighted spaces
$$
x^{\ell+(n-2)/2}L^2_\bl=x^{\ell-1}L^2_{g_0}
$$
to spaces with two additional orders of decay and two b-derivatives,
namely $x^{l'+1}\Hb^2$ (recall that we are using the $g_0$-density for
these spaces), which means the domain space has weight
\begin{equation}\label{eq:central-sc-interval}
l'=\ell-1\in \Big(-1-\frac{n-2}{2}, -1+\frac{n-2}{2}\Big).
\end{equation}
In general, for the simplicity of discussion, and as this covers
already the most interesting case, we assume that $\beta$ is
skew-adjoint and $\beta'$ is sufficiently small. If $\re\beta'>
\frac{\beta^2}{4}-\Big(\frac{n-2}{2}\Big)^2$ (which includes
$\beta=0$, $\beta'=0$), then
\eqref{eq:central-sc-interval} is replaced by
\begin{equation}\begin{aligned}\label{eq:central-sc-interval-1}
l'\in
\Big(-1+&\frac{\im\beta}{2}-\re\sqrt{-\frac{\beta^2}{4}+\Big(\frac{n-2}{2}\Big)^2+\beta'},
\\
&-1+\frac{\im\beta}{2}+\re\sqrt{-\frac{\beta^2}{4}+\Big(\frac{n-2}{2}\Big)^2+\beta'}\,\Big),
\end{aligned}\end{equation}
where the right hand side contains the sub-interval where $\beta'$ is
replaced by $\re\beta'$, and then the real part in front of the square
roots can be dropped. We call the interval on the right hand side of
\eqref{eq:central-sc-interval-1} the {\em
central interval for weights for the
scattering end}, and we denote it by $(\nu_-,\nu_+)$. We refer to
Remark~\ref{rem:limiting-absorption} for an explanation of the role of
this particular weight interval free from the negatives of the
imaginary parts of the indicial roots.

In Section~\ref{sec:normal} we rescale
$N_0(\hat P(\sigma))$ using the dilation invariance. This amounts to
introducing $X=x/|\sigma|$, $\hat\sigma=\sigma/|\sigma|$, in terms of which
\begin{equation*}\begin{aligned}
N_0(\hat P(\sigma))=\sigma^2 \Big(\Delta_{g_0}+&\beta
X^2\Big(XD_X+i\frac{n-2}{2}\Big)+\beta'X^2\\
&-2 \hat\sigma\Big(X^2D_X+i\frac{n-1}{2}X+\frac{\beta-\gamma}{2}X\Big)\Big)
\end{aligned}\end{equation*}
and $\Delta_0$ is the Laplacian of the exact conic metric
$\frac{dX^2}{X^4}+\frac{h_0}{X^2}$. Since
$\sigma^{-2}N_0(\hat P(\sigma))$ is homogeneous with respect to $X$
dilations, it can be globally Mellin transformed, and one can consider
this as simultaneously resolving the scattering end, $X=0$, as well as
the conic point, $X^{-1}=0$. The normalization of the Mellin transform
$$
(\cM u)(\taub,y)=\int_0^\infty X^{-i\taub} u(X,y)\,\frac{dX}{X}
$$
corresponding to the choice of $X$, as opposed to the choice of
$X^{-1}$,
$$
(\cM^c u)(\taub^c,y)=\int_0^\infty X^{i\taub^c} u(X,y)\,\frac{dX}{X},
$$
with superscript $c$ standing for the conic point,
involves taking the negative of the dual variable (as the
transforms use powers of $X$ vs.\ $X^{-1}$, so $\taub^c=-\taub$
identifies the transforms) hence the indicial roots
for the scattering end are the negatives of the indicial roots at the
conic point. When $\beta=0$ and $\beta'=0$, the central interval at the conic point is
$(1-\frac{n-2}{2},1+\frac{n-2}{2})$ as is familiar from analysis of
the Laplacian on spaces with conic singularities, and this interval indeed is the
negative of the range \eqref{eq:central-sc-interval}.

The version of the theorem corresponding to our more general
operators is:

\begin{thm}\label{thm:main-gen}
Suppose that
$\beta$ is skew-symmetric and $\re\beta'>
\frac{\beta^2}{4}-\Big(\frac{n-2}{2}\Big)^2$.
Suppose that $(\nu_-,\nu_+)$ is the central interval for weights at
the scattering end, $l'\in(\nu_-,\nu_+)$, see
\eqref{eq:central-sc-interval-1}. Suppose that
$P(0):\Hb^{\infty,l'}\to\Hb^{\infty,l'+2}$ has trivial nullspace, an
assumption independent of $l'$ in this range. Suppose also that either
$$
r>-1/2+\im(\beta+\gamma)/2,\qquad l<-1/2+\im(\beta-\gamma)/2,
$$
or
$$
r<-1/2+\im(\beta+\gamma)/2,\qquad l>-1/2+\im(\beta-\gamma)/2.
$$
Let
$$
\hat
P(\sigma)=e^{-i\sigma/x}P(\sigma)e^{i\sigma/x}.
$$

There exists
$\sigma_0>0$ such
that
$$
\hat P(\sigma):\{u\in\Hscbr^{s,r,l}:\ \hat P(\sigma)u\in \Hscbr^{s-2,r+1,l+1}\}\to\Hscbr^{s-2,r+1,l+1}
$$
is invertible for $0<|\sigma|\leq\sigma_0$, $\im\sigma\geq 0$, with this inverse being the $\pm i0$ resolvent of
$P(\sigma)$ corresponding to $\pm\re\sigma>0$, and we have the
estimate
\begin{equation*}\begin{aligned}
&\|(x+|\sigma|)^{\alpha}u\|_{\Hscbr^{s,r,l}}\leq C\|(x+|\sigma|)^{\alpha-1}\hat
P(\sigma)u\|_{\Hscbr^{s-2,r+1,l+1}}
\end{aligned}\end{equation*}
for
$$
\alpha\in\Big(l-\nu_+,l-\nu_-\Big).
$$
\end{thm}

\begin{rem}\label{rem:main-b-gen}
We note that Remark~\ref{rem:main-b} remains valid in the present
more general setting, including its proof; the only way the proof is
affected is via inputting the estimate of Theorem~\ref{thm:main-gen}
instead of Theorem~\ref{thm:main}.
\end{rem}

\section{Resolved b-algebra}\label{sec:resolved}
In this section we introduce the space $\Psibr^{m,l,\nu,\delta}(X)$ of resolved b-pseudodifferential operators and prove its
basic properties. We explicitly consider $\sigma\geq 0$ for notational
simplicity; for $\sigma\leq 0$ one simply replaces $\sigma$ by
$|\sigma|$ at various points below, while for $\sigma$ with
$\im\sigma\geq 0$ the blow up discussed below is that of $\Tb^*_{\pa
  X}X\times\{0\}$ in $\Tb^*X\times\{\sigma\in\Cx:\ \im\sigma\geq 0\}$.

We recall that the b-pseudodifferential algebra is discussed in detail
in Melrose's book \cite{Melrose:Atiyah}; the companion paper
\cite{Vasy:Limiting-absorption-lag} as well as \cite[Section~2]{Vasy:Zero-energy} have a
brief summary of its properties, while
\cite[Section~6]{Vasy:Minicourse} has a detailed presentation relating it to
H\"ormander's uniform pseudodifferential algebra \cite[Chapter~18.1]{Hor}.

\begin{figure}[ht]
\begin{center}
\includegraphics[width=100mm]{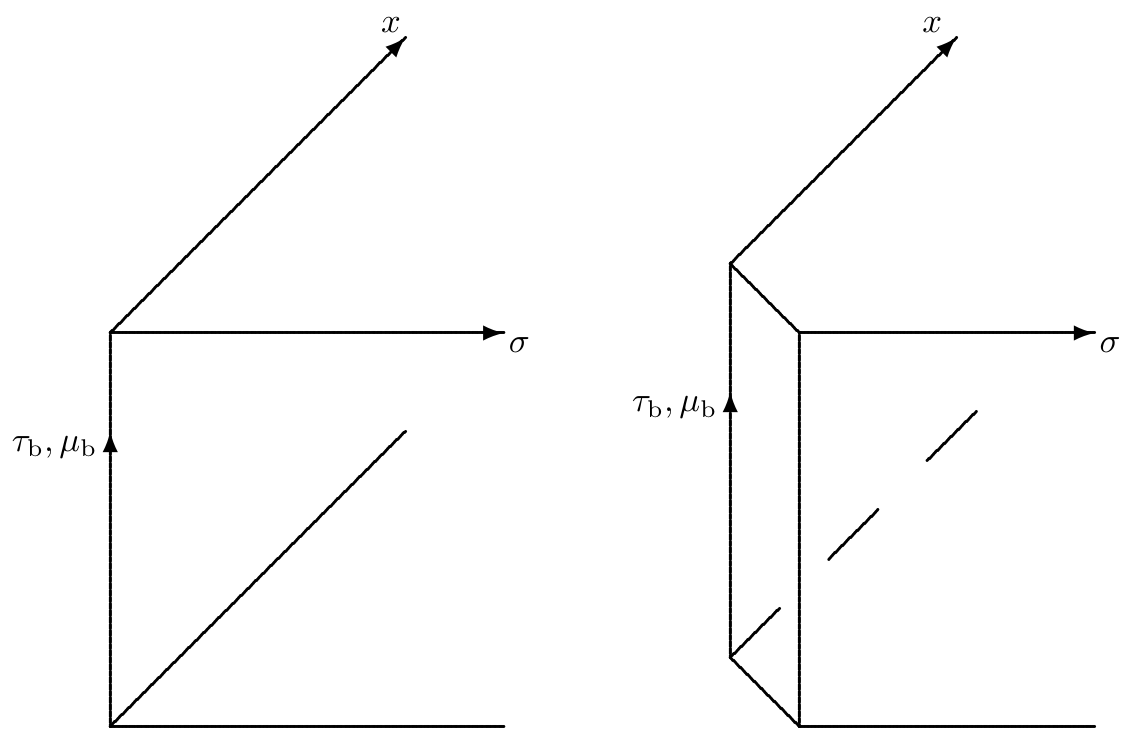}
\end{center}
\caption{The resolved b-cotangent bundle, on the right, obtained by blowing up the corner $\overline{\Tb^*}_{\pa
  X}X\times\{0\}$
  of $\overline{\Tb^*}X\times[0,1)_\sigma$, shown on the left.}
\label{fig:b-res}
\end{figure}

At the phase space level, the resolved algebra
simply blows up $\Tb^*X\times[0,1)_\sigma$ at the corner $\Tb^*_{\pa
  X}X\times\{0\}$; since the corner is given by $x=0$, $\sigma=0$, projectively this amounts to the introduction of
$x/\sigma$ and $\sigma/x$ as smooth variables, where bounded. In order
to work uniformly at fiber infinity, it is best to consider the blow
up
$$
\Tbr(X,[0,1))=[\overline{\Tb^*}X\times[0,1); \overline{\Tb^*}_{\pa X}X\times\{0\}]
$$
of $\overline{\Tb^*}_{\pa X}X\times\{0\}$ in
$\overline{\Tb^*}X\times[0,1)$, see Figure~\ref{fig:b-res}. Here $\overline{\Tb^*}X$ is the fiber radially
compactified b-cotangent bundle, i.e.\ the fibers are compactified, being
vector spaces, to balls, thus to manifolds with boundary, see
\cite{RBMSpec} for the compactification in the scattering setting, and
\cite[Section~5]{Vasy:Zero-energy} for a discussion connecting the
scattering and b-settings in the context of second microlocalization
at the zero section. Examples of defining functions of the
lift of $x=0$, the
front face (i.e.\ the lift of $x=0$, $\sigma=0$), resp.\ the lift of $\sigma=0$ are
$$
(1+\sigma/x)^{-1}=\frac{x}{x+\sigma},\ x+\sigma,\  \text{resp.}\
(1+x/\sigma)^{-1}=\frac{\sigma}{x+\sigma};
$$
a defining function of
fiber infinity is $(\taub^2+|\mub|^2)^{-1/2}$. Note that b-vector fields on the total
space lift to b-vector fields on its resolution since a boundary face
is being resolved; in particular
$x\pa_x$, $\pa_y$, $\tilde\rho^{-1}\pa_{\taub}$,
$\tilde\rho^{-1}\pa_{\mub}$, $\tilde\rho$ a defining function of fiber
infinity, thus $\tilde\rho^{-1}$ equivalent to the larger of
$\taub,\mub$, i.e.\ b-vector fields on the fibers over fixed $\sigma$,
lift to such, as does $\sigma\pa_\sigma$. Thus, a conormal family of
symbols on the resolved space is also a conormal family on the
original, unresolved, space, and thus can be regarded as a family of
b-symbols bounded by an appropriate power of $\sigma$, thus quantized,
etc. We write the symbol orders as
$$
S^{m,l,\nu,\delta}(\Tbr^*(X,[0,1))),
$$
where
$m$ is the b-differential order, $l$ is the order at (the lift of) $x=0$,
$\nu$ is the order at the front face (i.e.\ the lift of $x=\sigma=0$)
and $\delta$ is the order at (the lift of)
$\sigma=0$. Since $\sigma$ is a parameter (is commutative), one can
easily arrange that the last order is $\delta=0$, but it can be useful
to have some flexibility. A typical example of an elliptic symbol of
order $m,l,\nu,\delta$ is
then
$$
(\taub^2+\mub^2)^{m/2}(1+\sigma/x)^{l}(x+\sigma)^{-\nu}(1+x/\sigma)^{\delta}=(\taub^2+\mub^2)^{m/2}(x+\sigma)^{l-\nu+\delta}x^{-l}\sigma^{-\delta}.
$$
Correspondingly, the relationship between
$S^{m,l,\nu,\delta}(\Tbr^*(X,[0,1)))$ and the symbol space
$S^{m,\alpha,\beta}(\Tb^*(X;[0,1)))$ (symbols on $\Tb^*X\times[0,1)$,
of order $\alpha$ at $x=0$, $\beta$ at $\sigma=0$) is
\begin{equation}\label{eq:b-symbol-incl}
S^{m,l,\nu,\delta}(\Tbr^*(X,[0,1)))\subset
S^{m,\alpha,\beta}(\Tb^*(X;[0,1))),\ \alpha\leq l,\ \beta\leq \delta,\ \alpha+\beta\leq\nu.
\end{equation}
Explicitly, cf.\ \cite[Section~2]{Vasy:Zero-energy}, the quantization map, giving $A\in\Psibr^{m,l,\nu,\delta}(X)$, is
\begin{equation}\begin{aligned}\label{eq:b-quantization}
Au(x,y)=(2\pi)^{-n}\int
e^{i(\frac{x-x'}{x}\taub+(y-y')\mub)}&\tilde\psi\Big(\frac{x-x'}{x'}\Big)\\
&a(x,y,\taub,\mub,\sigma) u(x',y')\,d\taub\,d\mub\,\frac{dx'\,dy'}{x'},
\end{aligned}\end{equation}
with $\tilde\psi$ of compact support in $(-1/2,1/2)$, identically $1$
near $0$, which may be regarded as a member of
$\Psib^{m,\alpha,\beta}(X)$ with $\alpha,\beta$ as above, though this
is imprecise unless $\alpha=l$, $\beta=\delta$ and
$\alpha+\beta=\nu$, i.e.\ $\nu=l+\delta$.

It is useful to note
here that
$$
\frac{x'+\sigma}{x+\sigma}=\frac{x'-x}{x+\sigma}+1,
$$
and
$$
\Big|\frac{x'-x}{x+\sigma}\Big|\leq\Big|\frac{x'-x}{x}\Big|,
$$
with $\frac{x'-x}{x+\sigma}$ having the same sign as $\frac{x'-x}{x}$,
so over compact subsets of the {\em b-front face} (the lift of $x=x'=0$ to
the b-double space, which is the space resulting from blowing up this
submanifold, i.e.\ $\pa X\times\pa X$, in $X\times X$), where $\frac{x'-x}{x}$ is
in a compact subset of $(-1,\infty)$ (cf.\ $\tilde\psi$ in \eqref{eq:b-quantization}), $\frac{x'-x}{x+\sigma}$ is in
the same region. Thus, conjugating the localized in compact subsets of
the front face (family) b-algebra by powers of $x+\sigma$ is a
isomorphism, hence one can indeed work with the standard family
b-algebra with $(x+\sigma)^{\tilde\nu}$-weights. Notice that the full weight
$$
(1+\sigma/x)^{l}(x+\sigma)^{-\nu}(1+x/\sigma)^{\delta}
$$
can be rewritten in terms of a power of $x+\sigma$, times powers of
$x$ and $\sigma$, so as the latter two are well-behaved as far as the
conjugation is concerned, so is the total weight.

Indeed,
$$
x'<x\Rightarrow \frac{x'-x}{x}\leq\frac{x'-x}{x+\sigma}\leq 0
$$
and
$$
x'>x\Rightarrow 0\leq \frac{x'-x}{x+\sigma}\leq \frac{x'-x}{x}
$$
show that $\frac{x'+\sigma}{x+\sigma}$ is controlled by
$\frac{x'}{x}=\frac{x'-x}{x}+1$, so even in the small (family)
b-algebra, with infinite order vanishing on the side faces, i.e.\ as
$\frac{x'}{x}$, resp.\ $\frac{x}{x'}$, tend to $0$, so equivalently
the reciprocals $\frac{x}{x'}$, resp.\ $\frac{x'}{x}$, tend to $\infty$,
analogous statements hold.

It is useful to `complete' the resolved b-algebra by order $-\infty$ in the
differential sense terms. For this recall that for the unresolved
family, order $-\infty$ operators, in $\Psib^{-\infty,l}(X)$, have Schwartz kernels which are
conormal at the b-front face (of order $l$, interpreted as a right b-density) and vanish to infinite order at the left
and right faces ($x=0$, resp.\ $x'=0$, lifted to the b-double space),
with smooth (or conormal) behavior in $\sigma$. The new resolution is then
that of the corner given by the b-front face, locally defined by $x+x'$, at $\sigma=0$, which introduces
coordinates $\frac{x+x'}{\sigma},\sigma$ where $\sigma$ is relatively
large, and $\frac{\sigma}{x+x'},x+x'$  where $x+x'$ is such, together
with $\frac{x}{x'}$ or $\frac{x'}{x}$, as well as $y,y'$; see Figure~\ref{fig:b-double-res}. The class of distributions giving the
Schwartz kernel is conormal ones to all boundary hypersurfaces with
infinite order vanishing at the lifts of the left and right faces. Again, as a corner is being blown up, the
property of being conormal does not change, though orders are
affected, so the new class of operators is still a subclass of the
family b-pseudodifferential operators. Note that this blow up indeed
corresponds to the one at the symbol level: the quantization map for
the family \eqref{eq:b-quantization} gives a Schwartz kernel which is the inverse Fourier transform
in the fiber variables of the b-cotangent bundle to the fibers of the
b-double space
over a fixed point $(x,y)$ on the boundary, relative to the diagonal
$\frac{x-x'}{x}=0$, $y-y'=0$, localized by $\tilde\psi$: the blow-up
in the cotangent bundle, i.e.\ that of $x=\sigma=0$, commutes with the inverse
Fourier transform, considered as mapping to the local product space in
$x,y,\sigma,\frac{x-x'}{x},y-y'$. The conjugation invariance by powers of $x+\sigma$ then
follows from the previous paragraph.

\begin{figure}[ht]
\begin{center}
\includegraphics[width=120mm]{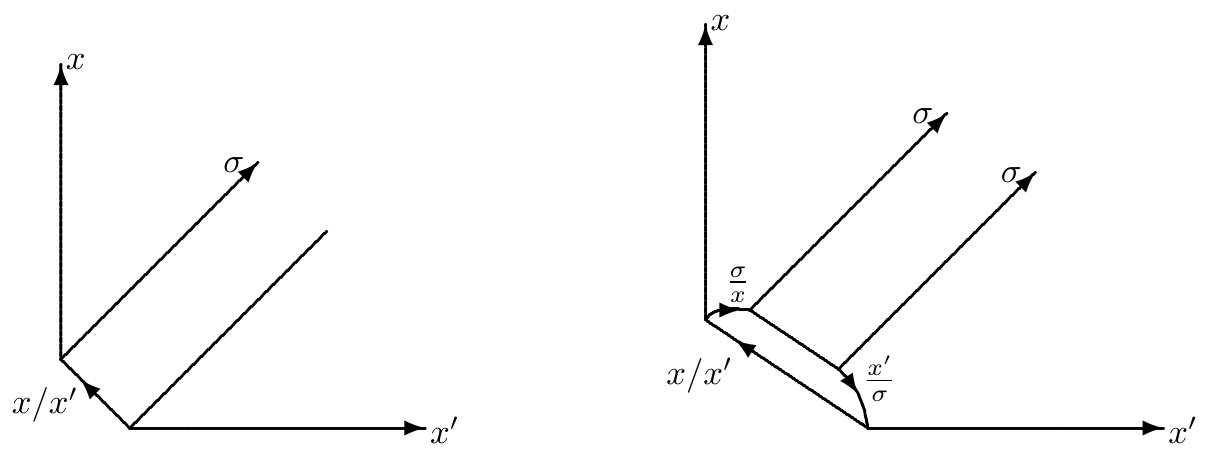}
\end{center}
\caption{The resolved b-double space, on the right, obtained by
  blowing up the corner given by the b-front face at $\sigma=0$
  of the b-double space times $[0,1)_\sigma$, shown on the left.}
\label{fig:b-double-res}
\end{figure}

The standard composition rules hold, including
full asymptotic expansions. Thus, using symbols on the resolved space,
one can define a resolved b-wave front set, $\WFbr'(A)$, resp.\
$\WFbr(u)$, of operator families, resp.\ distributions, with both
being a subset of fiber infinity of the resolved space, which we
denote by $\Sbr^*(X,[0,1))$, with the main point being that points
with various finite values of $x/\sigma$ are now distinguished for
$x=0$, $\sigma=0$. Finally, $A\in\Psibr^{m',l',l',0}(X)$ acts on the
standard b-Sobolev spaces since it lies in a continuous family of
b-operators in $\Psib^{m',l'}(X)$, cf.\ \eqref{eq:b-symbol-incl} taken
with $\beta=0$, giving estimates
$$
\|Au\|_{\Hb^{\tilde r-m',l-l'}}\leq C\|u\|_{\Hb^{\tilde r,l}},
$$
with $C$ independent of $\sigma$; more generally
$A\in\Psibr^{m',l',l'+k',k'}(X)$ gives estimates
$$
\|Au\|_{\Hb^{\tilde r-m',l-l'}}\leq C|\sigma|^{-k'}\|u\|_{\Hb^{\tilde r,l}},
$$
with uniform $C$.

\begin{figure}[ht]
\begin{center}
\includegraphics[width=120mm]{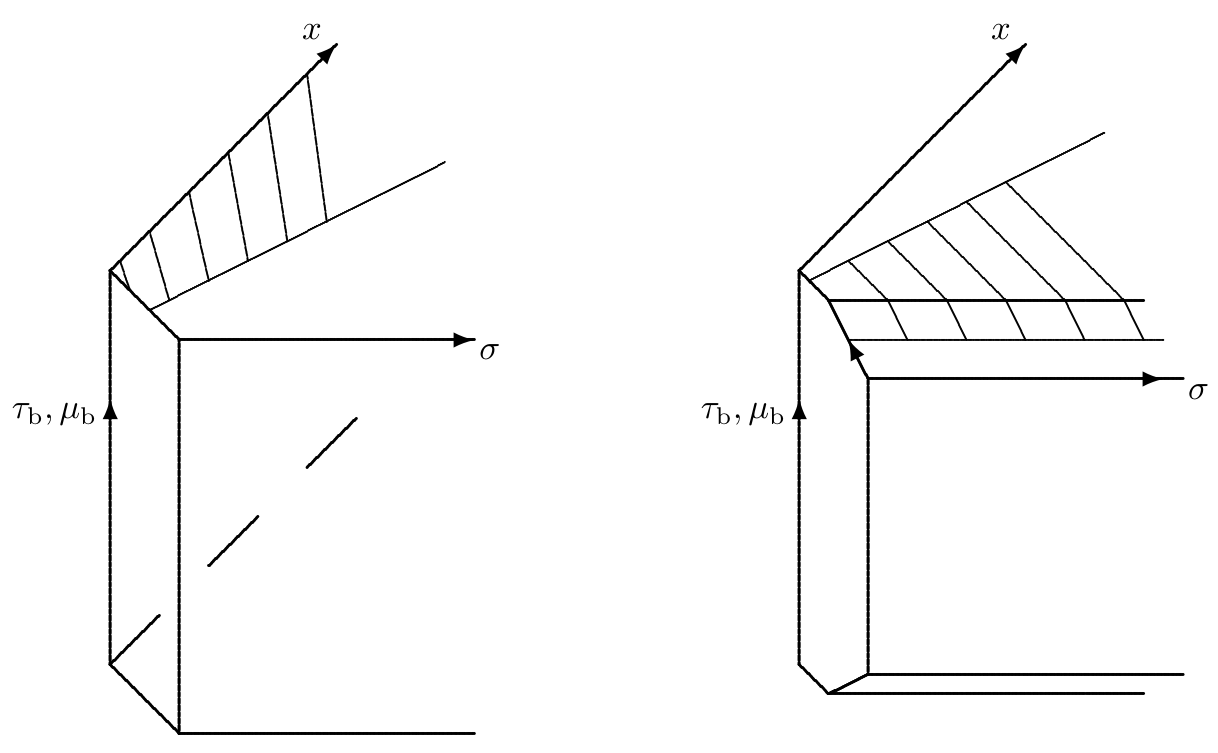}
\end{center}
\caption{Microsupport of the operator $B_1$ in \eqref{eq:br-fiber-infty} on the resolved
  b-cotangent bundle on the left, resp.\ the operator $B_1$ in \eqref{eq:scr-fiber-infty}
  on the
  scattering-b resolved cotangent bundle on the right. Both are shown
  as shaded regions.}
\label{fig:b-sc-res-supp}
\end{figure}

With this, $\hat P(\sigma)$ is
elliptic in $\Psibr^{2,-1,-2,0}(X)$ away from the lift of $x=0$ (i.e.\
away from $x/\sigma=0$), thus elliptic where
$x,S=\sigma/x$ (together with $y,\taub,\mub$) are valid coordinates, at
least for $x$ small, since its
principal symbol is
$x^2(\taub^2+\mub^2)$. Notice that this is a key advantage of working
with the resolved space: on the front face, both $\sigma x(xD_x)$ and
$x^2(xD_x)^2$ have the same decay order, $-2$ (i.e.\ $2$ orders of
decay), while in the decay sense the former dominates at $x/\sigma=0$
(order $-1$) and the latter at $\sigma/x=0$ (order $0$), though of
course only the latter matters in the standard principal symbol sense (order $2$). In particular, elliptic estimates hold in this region:
\begin{equation}\begin{aligned}\label{eq:br-fiber-infty}
&\|(1+x/\sigma)^\delta B_1 u\|_{\Hb^{\tilde r,l}}\\
&\qquad\leq C(\|(1+x/\sigma)^\delta (x+\sigma)^{-2}B_3\hat
P(\sigma)u\|_{\Hb^{\tilde r-2,l}}+\|(1+x/\sigma)^\delta
u\|_{\Hb^{-N,l}}),
\end{aligned}\end{equation}
with $B_1,B_3\in\Psibr^{0,0,0,0}(X)$ with wave front set away from the
lift of $x=0$ (i.e.\ away from $x/\sigma=0$), $B_3$ elliptic on a
neighborhood of $\WFbr'(B_1)$, and
where by a careful arrangement of support properties of $B_3$, one
could also replace $(x+\sigma)^{-2}$ by $x^{-2}$. See Figure~\ref{fig:b-sc-res-supp}.

\begin{figure}[ht]
\begin{center}
\includegraphics[width=120mm]{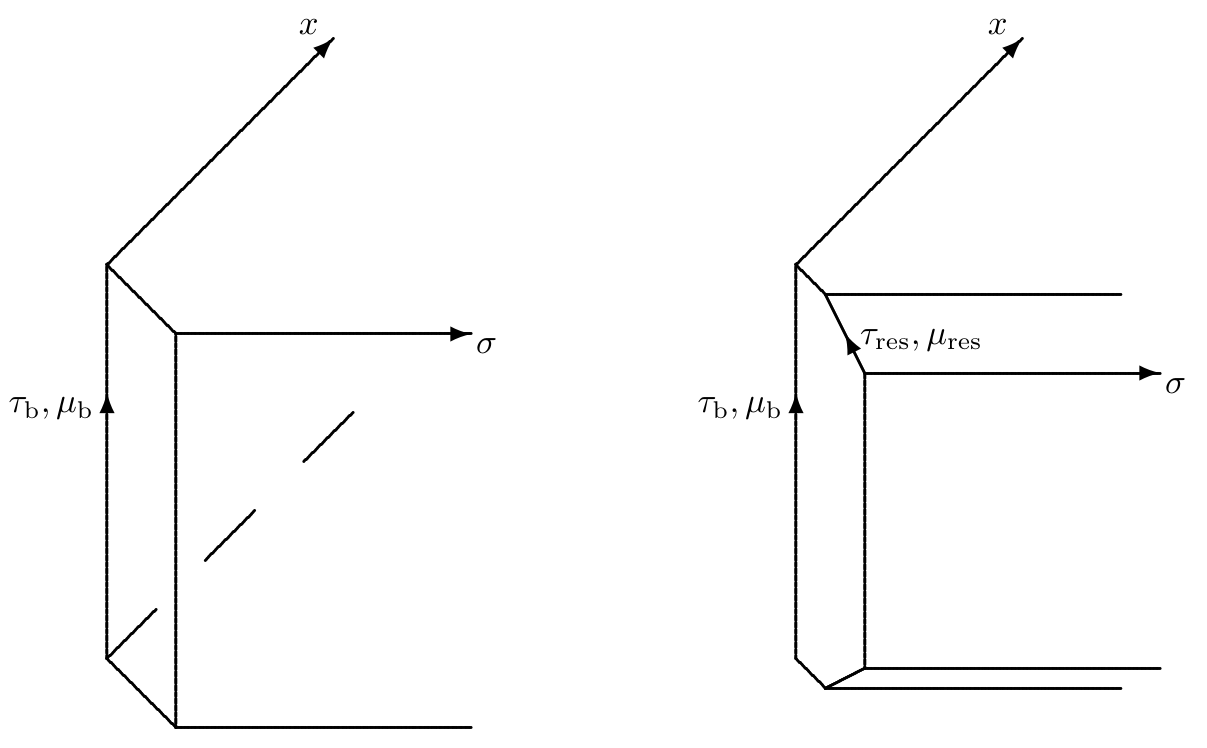}
\end{center}
\caption{The resolved b-cotangent bundle on the left, and its
  scattering-b resolution on the right obtained by blowing up the
  corner $x/\sigma=0$ at fiber infinity (nearest horizontal edges)
  of the resolved b-cotangent bundle. At the pseudodifferential
  operator level the symbolic calculus works at resolved b-fiber
  infinity which is the top (as well as bottom!) face on both pictures, as well as new face
on the right picture, which corresponds to rescaled sc-decay.}
\label{fig:b-sc-res}
\end{figure}

We in fact have
$$
\hat P(\sigma)\in\Psibr^{2,-2,-2,0}(X)+\Psibr^{1,-1,-2,0}(X),
$$
rather than merely in $\Psibr^{2,-1,-2,0}(X)$, which means that there
is a degeneracy at the lift of $x=0$, i.e.\ at $X=0$, where $X=x/\sigma$.
This is, however, fixed by second microlocal considerations, which
take the form of a resolution of fiber infinity at
$X=x/\sigma=0$, see Figure~\ref{fig:b-sc-res}. This introduces a (rescaled, by $\sigma$) scattering momentum
variable,
$$
\tausc=\taub X=\tau/\sigma,\ \musc=\mub X=\mu/\sigma
$$
in the interior of the new front face (via $|(\taub,\mub)|^{-1}$, $X$, $\taubh=\taub/|(\taub,\mub)|$,
$\mubh=\mub/|(\taub,\mub)$ being smooth nearby prior to the
blow-up, giving $|(\taub,\mub)|^{-1}/X$, $X$, $\taubh,\mubh$ smooth
after the blow-up, with the first quotient away from $0$, $\infty$ in
the interior).
Since this is a blow-up of a corner, the conormal spaces are
unchanged, but now one can allow different orders at fiber infinity
(which is now the sc-differential order) and at the new sc-front face,
which is the sc-decay order. One obtains
$\Psiscbr^{s,r,l,\nu,\delta}(X)$ this way, with
$$
(\tausc^2+\musc^2+1)^{s/2}((\taub^2+\mub^2)^{-1}+x^2/\sigma^2)^{-r/2}(1+\sigma/x)^{l}(x+\sigma)^{-\nu}(1+x/\sigma)^{\delta}
$$
being a typical elliptic symbol. Note that $\tausc,(\musc)_j$ are
the principal symbols of $\sigma^{-1} x^2D_x,\sigma^{-1}
xD_{y_j}\in\Psiscbr^{1,0,0,0,1}(X)$ (indeed in
$\Psiscbr^{1,0,-1,0,1}(X)$) which are singular as (non-resolved!) scattering
vector fields at $\sigma=0$.

This pseudodifferential space also gives rise to scattering-b-resolved Sobolev space
$\Hscbr^{s,r,l}(X)$ {\em family}, namely this is a family of Sobolev
spaces on $X$ depending on $\sigma$, which are the same as a topological vector space for
$\sigma\neq 0$, but with a $\sigma$-dependent norm. Before defining it in general, in the special
case when $r=s+l$, this is simply the b-Sobolev space family
(i.e.\ with elements depending on $\sigma$) $\Hb^{s,l}(X)$,
with norm defined independently of $\sigma$.
In general, it is the scattering-b-Sobolev space $\Hscb^{s,r,l}(X)$,
but with a $\sigma$-dependent norm:
\begin{equation}\label{eq:Hscbr-def}
\|u\|_{\Hscbr^{s,r,l}}^2=\|Au\|^2_{L^2}+\|u\|^2_{\Hb^{-N,l}},
\end{equation}
where $A\in\Psiscbr^{s,r,l,l,0}(X)$ is elliptic in the first two
senses (sc-differentiability and sc-decay) and
where $N$ is sufficiently large so that $s\geq -N$, $r\geq -N+l$.
In particular,
\begin{equation}\label{eq:Hscbr-Hb}
\Hscbr^{s,s+l,l}(X)=\Hb^{s,l}(X),
\end{equation}
as one can use an appropriate elliptic element $A$ of $\Psib^{s,l}(X)$ in
the definition in
this case.
With
this definition, we have for
$$
A\in\Psiscbr^{s',r',l',l'+k',k'}(X)=|\sigma|^{-k'}\Psiscbr^{s',r',l',l',0}(X)
$$
that
\begin{equation}\label{eq:Hscbr-bound}
\|Au\|_{\Hscbr^{s-s',r-r',l-l'}}\leq C|\sigma|^{-k'}\|u\|_{\Hscbr^{s,r,l}}.
\end{equation}
To give more feel for these spaces, in particular for the meaning of
the differential
order, we also remark that if $V\in\Vb(X)$, then
$\frac{x}{x+\sigma}V\in\Psiscbr^{1,0,-1,0,0}(X)\subset\Psiscbr^{1,0,0,0,0}(X)$,
and moreover at the resolved scattering fiber-infinity, corresponding
to the first order, at each point there is such a $V$ that is
elliptic. Thus, for instance $\|u\|_{\Hscbr^{s,r,l}}$ is an equivalent
norm to
\begin{equation}\label{eq:Hscbr-vf}
\sum_j\Big\|\frac{x}{x+\sigma}V_j u\Big\|_{\Hscbr^{s-1,r,l}}+\|u\|_{\Hscbr^{s-1,r,l}},
\end{equation}
where the $V_j$ span $\Vb(X)$ (so at each point one of them is
elliptic); in local coordinates one can take these as $xD_x,D_{y_k}$,
so roughly speaking, the differential regularity is in terms of
$\frac{x}{x+\sigma}xD_x,\frac{x}{x+\sigma}D_{y_k}$.

In part in order to become more familiar with these spaces, we make
some further remarks, in particular showing how the main
Theorem~\ref{thm:main} proves Remark~\ref{rem:main-b} and its
strengthened version.
Thus, if
$A\in\Psiscbr^{s-1,r+1,l+1,l+1,0}(X)$, then
$A(x+\sigma)^{-1}\in\Psiscbr^{s-1,r+1,l+1,l+2,0}(X)\subset\Psiscbr^{s-1,r+1,l+2,l+2,0}(X)$,
hence for such appropriate elliptic $A$ we have, using \eqref{eq:Hscbr-bound},
$$
\|(x+\sigma)^{-1} u\|_{\Hscbr^{s-1,r+1,l+1}}^2=\|A (x+\sigma)^{-1}
u\|_{L^2}^2+\|u\|_{\Hb^{-N,l+1}}^2\leq C\|u\|_{\Hscbr^{s-1,r+1,l+2}}^2.
$$
Therefore, Theorem~\ref{thm:main} gives,
with the first inequality being that of the theorem, for $s=r-l$,
\begin{equation}\begin{aligned}\label{eq:main-b-proof}
&\|(x+\sigma)^{\alpha}u\|_{\Hb^{s,l}}=
\|(x+\sigma)^{\alpha}u\|_{\Hscbr^{s,r,l}}\\
&\qquad\leq C\|(x+\sigma)^{\alpha-1}\hat
P(\sigma)u\|_{\Hscbr^{s-2,r+1,l+1}}\leq C\|(x+\sigma)^{\alpha-1}\hat
P(\sigma)u\|_{\Hscbr^{s-1,r+1,l+1}}\\
&\qquad\leq C'\|(x+\sigma)^{\alpha}\hat
P(\sigma)u\|_{\Hscbr^{s-1,r+1,l+2}}=C'\|(x+\sigma)^{\alpha}\hat
P(\sigma)u\|_{\Hb^{s-1,l+2}},
\end{aligned}\end{equation}
proving Remark~\ref{rem:main-b}.

This can be strengthened (made less lossy relative to the main
theorem) by combining the argument with
\eqref{eq:Hscbr-vf}. Thus, with $V_j$ as there, taking $s=r-l+1$, we have
\begin{equation*}\begin{aligned}
&\sum_j\Big\|\frac{x}{x+\sigma}V_j 
(x+\sigma)^{\alpha}u\Big\|_{\Hb^{s-1,l}}+\|(x+\sigma)^{\alpha}u\|_{\Hb^{s-1,l}}\\
&=\sum_j\Big\|\frac{x}{x+\sigma}V_j
(x+\sigma)^{\alpha}u\Big\|_{\Hscbr^{s-1,r,l}}+\|(x+\sigma)^{\alpha}u\|_{\Hscbr^{s-1,r,l}}\\
&\qquad\leq C'' \|(x+\sigma)^{\alpha}u\|_{\Hscbr^{s,r,l}}\\
&\qquad\leq C\|(x+\sigma)^{\alpha-1}\hat
P(\sigma)u\|_{\Hscbr^{s-2,r+1,l+1}}\\
&\qquad\leq C'\|(x+\sigma)^{\alpha}\hat
P(\sigma)u\|_{\Hscbr^{s-2,r+1,l+2}}\\
&\qquad=C'\|(x+\sigma)^{\alpha}\hat
P(\sigma)u\|_{\Hb^{s-2,l+2}},
\end{aligned}\end{equation*}
and in the first term $(x+\sigma)^\alpha$ can be commuted to the front
(up to changing constants) if one wishes.
In particular, on the left hand side, one can estimate $\|x V_j (x+\sigma)^{\alpha}
u\|_{\Hscb^{s-1,r,l}}$, which means
$\|(x+\sigma)^{\alpha}u\|_{\Hscb^{s,r,l}}=\|(x+\sigma)^{\alpha}u\|_{\Hscb^{s,s-1+l,l}}$, with $\Hscb$ the
standard second microlocal space (with $\sigma$-independent norm) as
in \cite[Section~5]{Vasy:Zero-energy}; this gains an extra
sc-derivative relative to Remark~\ref{rem:main-b}; it gives
\begin{equation}\begin{aligned}\label{eq:main-scb-proof}
\|(x+\sigma)^{\alpha}u\|_{\Hscb^{s,s+l-1,l}}\leq C \|(x+\sigma)^{\alpha}\hat
P(\sigma)u\|_{\Hscb^{s-2,s+l,l+2}}.
\end{aligned}\end{equation}

We now turn to our operator $\hat P(\sigma)$ and how it fits within our
resolved algebra.
In this scattering-b-resolved algebra we have
$$
\hat P(\sigma)\in\Psiscbr^{2,0,-1,-2,0}(X)
$$
with principal symbol in the first, sc-differential, sense at $X=0$ being
$$
x^2(\taub^2+\mub^2)=\sigma^2(\tausc^2+\musc^2),
$$
and the dual metric function in general (including away from $X=0$).
This is elliptic at the lift of fiber infinity at, thus near, $X=0$, where
$\sigma=0$ defines the base-front-face (the penultimate order),
$(\tausc^2+\musc^2)^{-1/2}$ the scattering fiber infinity (the first
order), while the third (b-decay) and last ($\sigma/x=0$ behavior)
orders are irrelevant. Now, even in the sc-decay sense, we have
ellipticity near sc-fiber-infinity, for in that sense the principal
symbol is
$$
x^2(\taub^2+\mub^2)-2\sigma x\taub=\sigma^2(\tausc^2+\musc^2-2\tausc)=\sigma^2((\tausc-1)^2+\musc^2-1),
$$
which is elliptic for sufficiently large $(\tausc,\musc)$. Thus, we
have microlocal elliptic estimates
\begin{equation}\label{eq:scr-fiber-infty}
\| B_1 u\|_{\Hscbr^{s,r,l}}\leq C(|\sigma|^{-2}\|B_3\hat
P(\sigma)u\|_{\Hscbr^{s-2,r,l}}+\|u\|_{\Hscbr^{-N,-N,-N}}),
\end{equation}
for $B_1$ microlocalizing in a neighborhood of sc-fiber infinity, made trivial near the lift of $X=0$
(the b-front face), with $B_3$ similar, but elliptic on the wave front
set of $B_1$; see Figure~\ref{fig:b-sc-res-supp}.

In combination these two elliptic estimates \eqref{eq:br-fiber-infty}-\eqref{eq:scr-fiber-infty} give
\begin{equation}\begin{aligned}\label{eq:scb-fiber-infty}
\|(1+x/\sigma)^\delta B_1 u\|_{\Hscbr^{s,r,l}}\leq C&(\|(1+x/\sigma)^\delta (x+\sigma)^{-2}B_3\hat
P(\sigma)u\|_{\Hscbr^{s-2,r,l}}\\
&\qquad\qquad+\|(1+x/\sigma)^\delta
u\|_{\Hb^{-N,l}}).
\end{aligned}\end{equation}

\section{Symbolic estimates}\label{sec:symbolic}
We now turn to symbolic estimates at
$X=x/\sigma=0$. This is a simple extension of the argument for the
limiting absorption principle as presented in \cite{Vasy:Limiting-absorption-lag}.

Since from the standard conjugated scattering picture \cite{Vasy:Limiting-absorption-lag} we already know
that the zero section has radial points, the only operator that can give
positivity microlocally in a symbolic commutator computation is the
weight. Recall that the actual positive commutator estimates utilize
the computation of
\begin{equation}\label{eq:twisted-comm-expr}
i(\hat P(\sigma)^*A-A\hat P(\sigma))=i(\hat P(\sigma)^*-\hat P(\sigma))A+i[\hat P(\sigma),A]
\end{equation}
with $A=A^*$, so for non-formally-self-adjoint $\hat P(\sigma)$ there is a
contribution from the skew-adjoint part
$$
\im \hat P(\sigma)=\frac{1}{2i}(P(\sigma)-P(\sigma)^*)
$$
of
$\hat P(\sigma)$, relevant for us when $\sigma$ is not real or when
$\sigma$ is real but $\beta,\gamma\neq 0$; here the notation `$\im \hat P(\sigma)$' is motivated by the
fact that its principal symbol is actually $\im \hat p(\sigma)$, with
$\hat p(\sigma)$ being the principal symbol of $\hat P(\sigma)$. It is actually
a bit better to rewrite this, with
$$
\re \hat P(\sigma)=\frac{1}{2}(P(\sigma)+P(\sigma)^*)
$$
denoting the
self-adjoint part of $\hat P(\sigma)$, as
\begin{equation}\label{eq:commutator-expr-8}
i(\hat P(\sigma)^*A-A\hat P(\sigma))=(\im \hat P(\sigma) A+A\im \hat P(\sigma))+i[\re \hat P(\sigma),A].
\end{equation}
If
$A\in\Psibr^{2\tilde r-1,2l+1,2\nu+2,-\infty}$, $\hat P(\sigma)\in\Psibr^{2,-1,-2,0}(X)$
implies that the second term (the commutator) is a priori in
$\Psibr^{2\tilde r,2l,2\nu,-\infty}$. Here we are setting the last
order to $-\infty$ since we are working near $x/\sigma=0$ (for away
from there we already have elliptic estimates!), so it plays
no role; this also means that one can simply use
$\sigma^{-2\nu+2l-1}$, resp.\ $\sigma^{-2\nu+2l}$,
as the weight capturing the front face behavior for $A$, resp.\ the
second term, if $x^{-2l-1}$, resp.\ $x^{-2l}$ is used
as the spatial weight.
Via the usual quadratic form argument this thus estimates
$u$ in $|\sigma|^{\nu-l} \Hb^{\tilde r,l}$ in terms of $|\sigma|^{\nu-l-1} \hat P(\sigma)u$ in $\Hb^{\tilde
  r-1,l+1}$, assuming non-degeneracy. 

However, we in fact have
$$
\hat P(\sigma)\in\Psibr^{2,-2,-2,0}(X)+\Psibr^{1,-1,-2,0}(X),
$$
which means that the second term of \eqref{eq:commutator-expr-8} (the
commutator) will be in
$$
\Psibr^{2\tilde r,2l-1,2\nu,-\infty}+\Psibr^{2\tilde r-1,2l,2\nu,-\infty},
$$
hence will degenerate as an element of $\Psibr^{2\tilde
  r,2l,2\nu,-\infty}$. This is fixed by second microlocal
considerations, namely considering
\begin{equation}\label{eq:symb-A-class}
A\in\Psibr^{2\tilde
  r-1,2l+1,2\nu+2,-\infty}(X)=\Psiscbr^{2\tilde r-1,2(\tilde
  r+l),2l+1,2\nu+2,-\infty}(X),
\end{equation}
and using that
$$
\hat P(\sigma)\in
\Psiscbr^{2,0,-1,-2,0}(X)
$$
so that the commutator lies in
$$
[\re\hat P(\sigma),A]\in \Psiscbr^{2\tilde r,2(\tilde
  r+l)-1,2l,2\nu,-\infty}(X),
$$
for the algebra is commutative to leading order in the first two
orders (namely sc-differentiability and sc-decay). In fact, we modify this
somewhat by taking an appropriate
$S\in\Psibr^{-1,0,0,0}(X)=\Psiscbr^{-1,-1,0,0,0}(X)$ and considering
\begin{equation}\label{eq:comm-S-expr}
i[\re\hat P(\sigma),A]+AS\hat P(\sigma)+\hat P(\sigma)^*SA \in\Psiscbr^{2\tilde r,2(\tilde
  r+l)-1,2l,2\nu,-\infty}(X),
\end{equation}
with the last two terms having principal symbol $2\re\hat p(\sigma)\hat
s a$, if $a$ is the principal symbol of $A$, and $\hat s$ of $S$,
where $S$ will be chosen in a manner that cancels an indefinite term
near the scattering zero section. We remark that from the second
microlocal perspective, the rescaled sc-differential order is
irrelevant in view of the elliptic estimate
\eqref{eq:scb-fiber-infty}, but a byproduct of the particular
choice of $s$ is that the principal symbol of the
commutator $[\re\hat P(\sigma),A]$ in the
sc-differential sense is also cancelled at $x=0$.

On the other hand, in general in the first
term
$$
\im \hat P(\sigma)\in \Psibr^{1,-1,-2,0}(X)=\Psiscbr^{1,0,-1,-2,0}(X),
$$
so the
first term of \eqref{eq:commutator-expr-8} is in $\Psibr^{2\tilde r,2l,2\nu,-\infty}(X)$, so is the
same order, $2l$, in the b-decay sense, as well as in the resolved
front face sense (order $2\nu$), as the modified commutator, but is actually bigger, order
$2(\tilde r+l)$, in scattering decay sense. However, when $\sigma$ is real,
then
$$
\im\hat P(\sigma)\in\Psiscbr^{1,-1,-1,-2,0}(X),
$$
so with $A$ as in \eqref{eq:symb-A-class}
$$
\im \hat P(\sigma) A+A\im \hat P(\sigma)\in\Psiscbr^{2\tilde r,2(\tilde
  r+l)-1,2l,2\nu,-\infty}(X),
$$
which has the same orders as $[\re\hat P(\sigma),A]$ and \eqref{eq:comm-S-expr}; we make some further adjustments to $S$ to
obtain a definite sign.

Now, going back to the issue of the zero section consisting of radial points,
we compute the principal symbol of the second term of \eqref{eq:twisted-comm-expr} (which is the only
term when $\sigma$ is real and $P(\sigma)=P(\sigma)^*$) when
$$
A \in\Psibr^{2\tilde r-1,2l+1,-2\tilde\nu+2l+1,-\infty}(X)
$$
is the
weight (as mentioned above, only this can give positivity) times a
cutoff in $x/\sigma$, i.e.
$$
x^{-2l-1}(x+\sigma)^{2\tilde\nu} (\taub^2+\mub^2)^{\tilde
  r-1/2}\phi(x/\sigma),\qquad \nu=-\tilde\nu+l-1/2,
$$
with $\phi\geq 0$ supported near $0$, identically $1$ in a smaller
neighborhood of $0$. Computationally it is better to take the equivalent, in view of the support
of $\phi$, 
$$
a=x^{-2l-1}\sigma^{2\tilde\nu} (\taub^2+\mub^2)^{\tilde
  r-1/2}\phi(x/\sigma),
$$
since $\sigma$ commutes with all operators --- the weight
$(x+\sigma)^{2\tilde\nu}$ has an equivalent effect as long
as $\supp\phi$ is taken sufficiently small.

\begin{lemma}
The principal symbol of
$$
(\im \hat P(\sigma) A+A\im \hat
P(\sigma))+i[\re \hat P(\sigma),A]\in\Psiscbr^{2\tilde r,2(\tilde r+l)-1,2l,-2\tilde\nu+2l-1,-\infty}(X)
$$
for real $\sigma$, suppressing
the $\phi(x/\sigma)$ factor as well as terms
involving its derivative, is
\begin{equation}\begin{aligned}\label{eq:commutator-skew-b-version}
x^{-2l}(\taub^2+\mub^2)^{\tilde
  r-3/2}\sigma^{2\tilde\nu}\Big(&4\sigma\Big(\Big(l+\tilde
r-\frac{\im(\beta-\gamma)}{2}\Big) \taub^2\\
&\qquad\qquad\qquad\qquad\qquad+\Big(l+1/2-\frac{\im(\beta-\gamma)}{2}\Big)\mub^2\Big)\\
&-4 x\Big(l+\tilde r-\frac{\im\beta}{2} \Big)\taub(\taub^2+\mub^2)\Big).
\end{aligned}\end{equation}
\end{lemma}

\begin{rem}
The cutoff factor $\phi(x/\sigma)$ contributes an additional term to the commutator, but as it is supported in
the elliptic region, this is estimated by the elliptic estimate, so
henceforth can be ignored. 
\end{rem}

\begin{proof}
Since the principal symbol of $\re
\hat P(\sigma)$ in the joint sc-differential-sc-decay sense is
$$
\re \hat p(\sigma)=x^2(\taub^2+\mub^2)-2x\re\sigma\taub=x^2(\taub^2+\mub^2)-2x\sigma\taub,
$$
we compute
\begin{equation}\begin{aligned}\label{eq:commutator-b-version-1}
&\{x^2(\taub^2+\mub^2)-2x\sigma\taub,x^{-2l-1}(\taub^2+\mub^2)^{\tilde
  r-1/2}\}\\
&=(2x^2\taub-2x\sigma)(-2l-1)
x^{-2l-1} (\taub^2+\mub^2)^{\tilde
  r-1/2}\\
&\qquad-(2x^2(\taub^2+\mub^2) -2x\sigma\taub) x^{-2l-1} 2(\tilde r-1/2)\taub (\taub^2+\mub^2)^{\tilde
  r-3/2}.
\end{aligned}\end{equation}
Expanding and rearranging,
\begin{equation}\begin{aligned}\label{eq:commutator-b-version-2}
&=4\sigma x^{-2l}(l+1/2)(\taub^2+\mub^2)^{\tilde
  r-1/2}\\
&\qquad+4\sigma(\tilde r-1/2)x^{-2l}\taub^2 (\taub^2+\mub^2)^{\tilde
  r-3/2}\\
&\qquad-4x^{-2l+1}(l+1/2)\taub (\taub^2+\mub^2)^{\tilde
  r-1/2}\\
&\qquad-4(\tilde r-1/2)x^{-2l+1}\taub (\taub^2+\mub^2)^{\tilde
  r-1/2}\\
&=x^{-2l} (\taub^2+\mub^2)^{\tilde
  r-3/2}\Big(4\sigma\Big((l+1/2) (\taub^2+\mub^2)+(\tilde
r-1/2)\taub^2\Big)\\
&\qquad\qquad\qquad\qquad-4x(l+\tilde r )\taub(\taub^2+\mub^2)\Big)\\
&=x^{-2l}(\taub^2+\mub^2)^{\tilde
  r-3/2}\Big(4\sigma\Big((l+\tilde r) \taub^2+(l+1/2)\mub^2\Big)\\
&\qquad\qquad\qquad\qquad\qquad-4x(l+\tilde r )\taub(\taub^2+\mub^2)\Big).
\end{aligned}\end{equation}

On the other hand, we have an additional term $\im \hat P(\sigma)
A+A\im \hat P(\sigma)$; by Proposition~\ref{prop:full-hat-P-sigma} its principal symbol for real $\sigma$ is
\begin{equation*}\begin{aligned}
&(2x\taub\im\beta-2\sigma\im(\beta-\gamma)) xa\\
&=2x^{-2l+1}\taub(\im\beta)\sigma^{2\tilde\nu} (\taub^2+\mub^2)^{\tilde
  r-1/2}\phi(x/\sigma)\\
&\qquad\qquad-2\im(\beta-\gamma)x^{-2l}\sigma^{2\tilde\nu+1} (\taub^2+\mub^2)^{\tilde
  r-1/2}\phi(x/\sigma).
\end{aligned}\end{equation*}

Thus, the total expression, suppressing the
$\phi(x/\sigma)$ factor as well as terms with its derivatives,
\begin{equation*}\begin{aligned}
x^{-2l}(\taub^2+\mub^2)^{\tilde
  r-3/2}\sigma^{2\tilde\nu}\Big(&4(\re\sigma)\Big(\Big(l+\tilde
r-\frac{\im(\beta-\gamma)}{2}\Big) \taub^2\\
&\qquad\qquad\qquad\qquad\qquad+\Big(l+1/2-\frac{\im(\beta-\gamma)}{2}\Big)\mub^2\Big)\\
&-4 x\Big(l+\tilde r-\frac{\im\beta}{2} \Big)\taub(\taub^2+\mub^2)\Big),
\end{aligned}\end{equation*}
proving the lemma.
\end{proof}

\begin{rem}\label{rem:regularizer-choice}
In an analogue of Remark~4.2 of \cite{Vasy:Limiting-absorption-lag},
we record the impact of having an additional
regularizer factor, namely replacing $a$ by
$$
a^{(\ep)}=a f_\ep,
$$
as is standard  in positive
commutator estimates, including at radial points, see the references in
\cite{Vasy:Limiting-absorption-lag}. The slightly delicate issue at
radial points is the limitation of regularizability, which was not a problem
in \cite{Vasy:Limiting-absorption-lag} since we work in a
neighborhood of the radial set at the zero section there, and the
second microlocal setup means that the only potential issue (from the perspective
of regularization actually needed and limited from the purely scattering, as opposed to the
second microlocal, setting) amounts to
getting additional b-decay, which is irrelevant for the symbolic
considerations. However, in the present setting the two radial sets
are simultaneously considered, naturally in view of the $\sigma\to
0$ limit (since the radial sets `collide' in the limit), and thus the
limitations of regularizability are relevant.

As in \cite[ Remark~4.2]{Vasy:Limiting-absorption-lag},
we can take the regularizer of the form
$$
f_\ep(\taub^2+\mub^2),\ f_\ep(s)=(1+\ep s)^{-K/2},
$$
where $K>0$ fixed and $\ep\in[0,1]$, with the interesting
behavior being the $\ep\to 0$ limit. Note that $f_\ep(\taub^2+\mub^2)$
is a symbol of order $-K$ for $\ep>0$, but is only uniformly
bounded in symbols of order $0$, converging to $1$ in symbols of
positive order. Then
$$
sf'_\ep(s)=-\frac{K}{2} \frac{\ep s}{1+\ep s} f_\ep(s),
$$
and $0\leq \frac{\ep s}{1+\ep s}\leq 1$, so in particular
$sf'_\ep(s)/f_\ep(s)$ is bounded. {\em Just as in \cite[ Remark~4.2]{Vasy:Limiting-absorption-lag}}, the effect of this is to add an
overall factor of $f_\ep(\taub^2+\mub^2)$ to
\eqref{eq:commutator-skew-b-version}
and \eqref{eq:commutator-b-version-1} as well as the subsequent
expressions in the above proof, and replace every occurrence of
$\tilde r$, other than those in the exponent, by
\begin{equation}\label{eq:tilde-r-replace}
\tilde r+(\taub^2+\mub^2)
\frac{f'_\ep(\taub^2+\mub^2)}{f_\ep(\taub^2+\mub^2)}=\tilde r-\frac{K}{2}\frac{\ep(\taub^2+\mub^2)}{1+\ep(\taub^2+\mub^2)}.
\end{equation}
\end{rem}

It is useful to rewrite the last term in the outermost parentheses of
\eqref{eq:commutator-skew-b-version} since it comes with an indefinite
sign due to the factor of $\taub$.

\begin{lemma}\label{lemma:modified-commutator-real-princ}
Let
$$
\hat s=2\Big(l+\tilde r-\frac{\im\beta}{2} \Big)
(\taub^2+\mub^2)^{-1} \taub,
$$
and let
$S\in\Psiscbr^{-1,-1,0,0,0}(X)=\Psibr^{-1,0,0,0}(X)$ with principal
symbol $\hat s$. Then for real $\sigma$ the principal symbol of
\begin{equation*}\begin{aligned}
&i(\hat P(\sigma)^*A-A\hat P(\sigma))+AS\hat P(\sigma)+\hat
P(\sigma)^*SA\\
&=(\im \hat P(\sigma) A+A\im \hat P(\sigma))+i[\re \hat
P(\sigma),A]+AS\hat P(\sigma)+\hat P(\sigma)^*SA\\
&\qquad\qquad\qquad\in\Psiscbr^{2\tilde r,2(\tilde r+l)-1,2l,-2\tilde\nu+2l-1,-\infty}(X)
\end{aligned}\end{equation*}
at $x=0$
is, suppressing the factor $\phi(x/\sigma)$ as
well as terms with its derivatives,
\begin{equation}\begin{aligned}\label{eq:modified-commutator-real-princ}
&\sigma^{2\tilde\nu+1} x^{-2l}(\taub^2+\mub^2)^{\tilde
  r-3/2}4\Big(-\Big(l+\tilde
r-\frac{\im(\beta+\gamma)}{2}\Big) \taub^2\\
&\qquad\qquad\qquad\qquad\qquad\qquad\qquad\qquad+\Big(l+1/2-\frac{\im(\beta-\gamma)}{2}\Big)\mub^2\Big).
\end{aligned}\end{equation}
\end{lemma}

\begin{rem}
See \cite[Remark~4.6]{Vasy:Limiting-absorption-lag} for a discussion
of this choice of $\hat s$, including both the advantages and the
disadvantages, in the context of the more general operators considered
in \cite{Vasy:Limiting-absorption-lag}.

Moreover, the analogue of the conclusion remains valid with a regularizer as in
Remark~\ref{rem:regularizer-choice}, i.e.\ $a$ replaced by
$a^{(\ep)}$, and correspondingly $A$ by $A^{(\ep)}$, provided in the definition of $\hat s$ as well as in
the conclusion, $\tilde r$ is
replaced by \eqref{eq:tilde-r-replace} (except in the exponent), and in the conclusion an
overall factor of $f_\ep(\taub^2+\mub^2)$ is added.
\end{rem}

\begin{proof}
We add to $\sigma^{2\tilde\nu}$ times
\eqref{eq:commutator-b-version-1} the product of
\begin{equation*}\begin{aligned}
&2\hat s a=4\Big(l+\tilde r-\frac{\im\beta}{2} \Big)
(\taub^2+\mub^2)^{-1} \taub a\\
&=4 \sigma^{2\tilde\nu} x^{-2l-1}(\taub^2+\mub^2)^{\tilde
  r-3/2} \Big(l+\tilde r-\frac{\im\beta}{2} \Big)\taub\phi(x/\sigma)
\end{aligned}\end{equation*}
and the principal
symbol of $\re\hat P(\sigma)$, namely
$$
x^2(\taub^2+\mub^2)-2\sigma x\taub.
$$
We obtain, dropping the factor $\phi$,
\begin{equation*}\begin{aligned}
&\sigma^{2\tilde\nu} x^{-2l}(\taub^2+\mub^2)^{\tilde
  r-3/2}\Big(4\sigma\Big(\Big(l+\tilde r-\frac{\im(\beta-\gamma)}{2}\Big)
\taub^2\\
&\qquad\qquad\qquad\qquad\qquad\qquad\qquad\qquad+\Big(l+1/2-\frac{\im(\beta-\gamma)}{2}\Big)\mub^2\Big)\\
&\qquad\qquad\qquad\qquad\qquad\qquad-8\Big(l+\tilde r-\frac{\im\beta}{2}
\Big)\sigma\taub^2\Big)\\
=&\sigma^{2\tilde\nu+1} x^{-2l}(\taub^2+\mub^2)^{\tilde
  r-3/2}4\Big(-\Big(l+\tilde
r-\frac{\im(\beta+\gamma)}{2}\Big) \taub^2\\
&\qquad\qquad\qquad\qquad\qquad\qquad\qquad\qquad+\Big(l+1/2-\frac{\im(\beta-\gamma)}{2}\Big)\mub^2\Big),
\end{aligned}\end{equation*}
proving the lemma.
\end{proof}

For $l,\tilde r$ with $l+\tilde r-\frac{\im(\beta+\gamma)}{2}>0$,
$l+1/2-\frac{\im(\beta-\gamma)}{2}<0$, or with both terms having the opposite sign, we thus obtain
a positive commutator estimate.

\begin{prop}\label{prop:symbolic-est-real}
We have
\begin{equation*}\begin{aligned}
&\|(1+x/\sigma)^{\alpha}u\|_{\Hscbr^{s,\tilde r+l-1/2,l}}\\
&\leq C(\|(1+x/\sigma)^{\alpha}(x+\sigma)^{-1}\hat
P(\sigma)u\|_{\Hscbr^{s-2,\tilde r+l+1/2,l+1}}+\|(1+x/\sigma)^{\alpha}u\|_{\Hscbr^{-N,-N,l}}),
\end{aligned}\end{equation*}
provided that $l+\tilde r-\frac{\im(\beta+\gamma)}{2}>0$,
$l+1/2-\frac{\im(\beta-\gamma)}{2}<0$, or vice versa.

The estimate is valid in the sense that if $(1+x/\sigma)^{\alpha}u\in
\Hscbr^{s',\tilde r'+l-1/2,l}$ for some $s',\tilde r'$ with $\tilde
r'$ satisfying the inequality in place of $\tilde r$, and with $(1+x/\sigma)^{\alpha}(x+\sigma)^{-1}\hat
P(\sigma)u \in\Hscbr^{s-2,\tilde r+l+1/2,l+1}$ then $(1+x/\sigma)^{\alpha}u\in
\Hscbr^{s,\tilde r+l-1/2,l}$ and the estimate holds.
\end{prop}

\begin{proof}
At first we discuss the argument for sufficiently regular
$u$. Concretely, $u$ with $(1+x/\sigma)^{\alpha}u\in
\Hscbr^{s,\tilde r+l-1/2,l}$ suffices (so the left hand side is a
priori finite). Indeed, even in this case there is one subtlety, for
at first sight half an order additional regularity is needed, see the
proof of Proposition~4.10~\cite{Vasy:Limiting-absorption-lag}, as well
as the references given there, \cite[Proof of
Proposition~5.26]{Vasy:Minicourse} and
\cite[Lemma~3.4]{Haber-Vasy:Radial}, to make sense of $\hat
P(\sigma)^*A$ applied to $u$ and paired with $u$, but this is easily
overcome by a simple regularization argument (which does {\em not}
have limitations unlike the more serious regularization discussed below) given in the references.

Recalling that
$$
A
\in\Psibr^{2\tilde
  r-1,2l+1,-2\tilde\nu+2l+1,-\infty}(X)=\Psiscbr^{2\tilde r-1,2(\tilde
  r+l),2l+1,-2\tilde\nu+2l+1,-\infty}(X),
$$
where we take $\tilde\nu=\alpha-1/2$,
Lemma~\ref{lemma:modified-commutator-real-princ}
gives, with $S\in\Psiscbr^{-1,-1,0,0,0}(X)=\Psibr^{-1,0,0,0}(X)$ with principal symbol $\hat s$,
that
\begin{equation}\begin{aligned}\label{eq:total-twisted-comm-S}
&i(\hat P(\sigma)^*A-A\hat P(\sigma))+AS\hat P(\sigma)+\hat
P(\sigma)^*SA\\
&=(\im \hat P(\sigma) A+A\im \hat P(\sigma))+i[\re \hat
P(\sigma),A]+AS\hat P(\sigma)+\hat P(\sigma)^*SA=\pm B^*B+F
\end{aligned}\end{equation}
in $\Psiscbr^{2\tilde r,2(\tilde
  r+l)-1,2l,-2\tilde\nu+2l-1,-\infty}(X)$
with
\begin{equation*}\begin{aligned}
&B\in\Psiscbr^{\tilde r-1/2,\tilde
  r+l-1/2,l,l-\tilde\nu-1/2,-\infty}(X)=\Psibr^{\tilde
  r-1/2,l,l-\tilde\nu-1/2,-\infty}(X),\\
& F\in\Psiscbr^{2\tilde r, 2\tilde r+2l-2,2l,2l-2\tilde\nu-1,-\infty}(X),
\end{aligned}\end{equation*}
and with the principal symbol of $B$ given by, up to a factor
involving $\phi$, by the square root of \eqref{eq:modified-commutator-real-princ}.
Note that $F$ only drops an order in the sc-decay relative to
\eqref{eq:total-twisted-comm-S}, but is actually higher order in the
sc-differentiability sense as the symbolic computation was performed
at $x=0$ and as the principal symbol in the sc-decay sense vanishes
there but not nearby. However, we already have elliptic estimates at sc-fiber
infinity, so this is of no relevance; by virtue of \eqref{eq:scb-fiber-infty} we have
\begin{equation*}\begin{aligned}
|\langle Fu,u\rangle|&\leq C'\|\sigma^{\tilde
  \nu+1/2}u\|^2_{\Hscbr^{\tilde r,\tilde r+l-1,l}}\\
&\leq C(|\sigma|^{-2}\|\sigma^{\tilde
  \nu+1/2}\hat P(\sigma)u\|^2_{\Hscbr^{\tilde r-2,\tilde r+l-1,l}}+\|\sigma^{\tilde
  \nu+1/2}u\|^2_{\Hscbr^{-N,\tilde r+l-1,l}})\\
&\leq C(|\sigma|^{-2}\|\sigma^{\tilde
  \nu+1/2}\hat P(\sigma)u\|^2_{\Hb^{\tilde r-1,l}}+\|\sigma^{\tilde
  \nu+1/2}u\|^2_{\Hscbr^{-N,\tilde r+l-1,l}}).
\end{aligned}\end{equation*}
Thus, computing \eqref{eq:total-twisted-comm-S} applied to $u$ and
paired with $u$ yields
\begin{equation*}\begin{aligned}
&\pm\|Bu\|^2+\langle Fu,u\rangle=2\im\langle \hat
P(\sigma)u,Au\rangle+2\re\langle \hat P(\sigma)u,SAu\rangle,
\end{aligned}\end{equation*}
so
$$
\|\sigma^{\tilde\nu+1/2}B_1 u\|_{\Hb^{\tilde r-1/2,l}}^2\leq C(\|\sigma^{\tilde\nu-1/2} B_3 \hat
P(\sigma)u\|_{\Hb^{\tilde r-1/2,l+1}}^2+\|\sigma^{\tilde\nu+1/2}u\|_{\Hb^{-N,l}}^2)
$$
with $B_1,B_3\in\Psibr^{0,0,0,0}(X)=\Psiscbr^{0,0,0,0,0}(X)$
microsupported near $x/\sigma=0$. In fact, one can even use a cutoff
with differential
supported near sc-fiber infinity, i.e.\ localizing to a compact region
in $(\tausc,\musc)$; this again gives an error term we can already
estimate by elliptic estimates.

Thus, combined with the existing elliptic estimates, this proves the
proposition in the weaker sense of a priori having a sufficiently
regular $u$.

In order to obtain the full result, we need to regularize, replacing
$A$ by $A^{(\ep)}$. The main impact of this is that in
\eqref{eq:modified-commutator-real-princ} $\tilde r$ is replaced by
\eqref{eq:tilde-r-replace} (except in the exponent) and an overall factor of
$f_\ep(\taub^2+\mub^2)$ is added, so that
\eqref{eq:modified-commutator-real-princ} is replaced by
\begin{equation}\begin{aligned}\label{eq:modified-commutator-real-princ-reg}
&\sigma^{2\tilde\nu+1} x^{-2l}(\taub^2+\mub^2)^{\tilde
  r-3/2} 4\Big(-\Big(l+\tilde
r-\frac{\im(\beta+\gamma)}{2}-\frac{K}{2}\frac{\ep(\taub^2+\mub^2)}{1+\ep(\taub^2+\mub^2)}\Big) \taub^2\\
&\qquad\qquad\qquad\qquad\qquad\qquad\qquad\qquad+\Big(l+1/2-\frac{\im(\beta-\gamma)}{2}\Big)\mub^2\Big)
f_\ep(\taub^2+\mub^2).
\end{aligned}\end{equation}
Here we need $K=2(\tilde r-\tilde r')$: the regularized operator
improves $K$ b-differentiability, thus sc-decay and
sc-differentiability, orders, but in the quadratic form both slots
have a $u$ which
needs improved regularity. Since $0\leq
\frac{\ep(\taub^2+\mub^2)}{1+\ep(\taub^2+\mub^2)}\leq 1$, this still
gives the desired definite sign, and the rest of the argument can
proceed essentially unchanged. We refer to
\cite[Section~5.4.7]{Vasy:Minicourse}, \cite[Proof of
Proposition~2.3]{Vasy-Dyatlov:Microlocal-Kerr},
as well as earlier work going
back to \cite{RBMSpec} and including
\cite[Theorem~1.4]{Haber-Vasy:Radial} for the concrete implementation. 
\end{proof}

We now turn to the case of not necessarily real $\sigma$.  {\em We
  remark that the regularization issues and the ways of dealing with them are completely analogous to
  the real $\sigma$ case, and we will not comment on these
  explicitly.}

Near the scattering zero section the term $-2\sigma\tau$
is the most important part of the principal symbol since the other
terms vanish quadratically at the zero section, so it is useful to
consider
$$
\tilde P(\sigma)=\sigma^{-1}\hat P(\sigma) \in\Psiscbr^{2,0,-1,-1,1}(X)
$$
so the principal symbol is
$$
\tilde p(\sigma)=\sigma^{-1}\hat p(\sigma)=-2\tau+\overline{\sigma}|\sigma|^{-2}(\tau^2+\mu^2),
$$
hence
$$
\re\tilde p(\sigma)=-2\tau+(\re\sigma)|\sigma|^{-2}(\tau^2+\mu^2),
$$
$$
\im\tilde p(\sigma)=-(\im\sigma)|\sigma|^{-2}(\tau^2+\mu^2).
$$
Thus, $\im\tilde p(\sigma)\leq 0$ if $\im\sigma\geq 0$, which means
one can propagate estimates forwards along the Hamilton flow of
$\re\tilde p(\sigma)$; similarly, if $\im\sigma\leq 0$, one can
propagate estimates backwards along the Hamilton flow of $\re\tilde
p(\sigma)$. We have

\begin{lemma}
We have
\begin{equation}\begin{aligned}\label{eq:commutator-sc-version-imag}
&H_{\re\tilde p(\sigma)}x^{-2l-1}(\taub^2+\mub^2)^{\tilde
  r-1/2}\\
&=x^{-2l}(\taub^2+\mub^2)^{\tilde
  r-3/2}\Big(4\big((l+\tilde r)
\taub^2+(l+1/2)\mub^2\big)-4\frac{\re\sigma}{|\sigma|^2} x(l+\tilde
r)\taub(\taub^2+\mub^2)\Big)\\
&=x^{-2(l+\tilde r)+1}(\tau^2+\mu^2)^{\tilde
  r-3/2}\Big(4\big((l+\tilde r) \tau^2+(l+1/2)\mu^2\big)-4\frac{\re\sigma}{|\sigma|^2} (l+\tilde r)\tau(\tau^2+\mu^2)\Big).
\end{aligned}\end{equation}
\end{lemma}

\begin{proof}
We compute
\begin{equation*}\begin{aligned}
&\Big\{\frac{\re\sigma}{|\sigma|^2} x^2(\taub^2+\mub^2)-2x\taub,x^{-2l-1}(\taub^2+\mub^2)^{\tilde
  r-1/2}\Big\}\\
&=(2\frac{\re\sigma}{|\sigma|^2} x^2\taub-2x)(-2l-1) x^{-2l-1} (\taub^2+\mub^2)^{\tilde
  r-1/2}\\
&\qquad-(2\frac{\re\sigma}{|\sigma|^2} x^2(\taub^2+\mub^2) -2x\taub) x^{-2l-1}2(\tilde r-1/2)\taub (\taub^2+\mub^2)^{\tilde
  r-3/2}.
\end{aligned}\end{equation*}
Expanding and rearranging,
\begin{equation*}\begin{aligned}
&=4(l+1/2) x^{-2l}(\taub^2+\mub^2)^{\tilde
  r-1/2}\\
&\qquad+4(\tilde r-1/2)x^{-2l}\taub^2 (\taub^2+\mub^2)^{\tilde
  r-3/2}\\
&\qquad-4\frac{\re\sigma}{|\sigma|^2} (l+1/2)x^{-2l+1}\taub (\taub^2+\mub^2)^{\tilde
  r-1/2}\\
&\qquad-4\frac{\re\sigma}{|\sigma|^2} (\tilde r-1/2)x^{-2l+1}\taub (\taub^2+\mub^2)^{\tilde
  r-1/2}\\
&=x^{-2l}(\taub^2+\mub^2)^{\tilde
  r-3/2}\Big(4\big((l+1/2) (\taub^2+\mub^2)+(\tilde
r-1/2)\taub^2\big)\\
&\qquad\qquad \qquad\qquad \qquad\qquad \qquad\qquad-4\frac{\re\sigma}{|\sigma|^2} x(l+\tilde r)\taub(\taub^2+\mub^2)\Big)\\
&=x^{-2l}(\taub^2+\mub^2)^{\tilde
  r-3/2}\Big(4\big((l+\tilde r) \taub^2+(l+1/2)\mub^2\big)-4\frac{\re\sigma}{|\sigma|^2} x(l+\tilde r)\taub(\taub^2+\mub^2)\Big).
\end{aligned}\end{equation*}
Writing $\tau=x\taub$, $\mu=x\mub$ proves the lemma.
\end{proof}

\begin{lemma}
Let
$$
\hat s_0=2(l+\tilde r)(\taub^2+\mub^2)^{-1} \taub,
$$
and let
$S_0\in\Psiscbr^{-1,-1,0,0,0}(X)=\Psibr^{-1,0,0,0}(X)$ with principal
symbol $\hat s_0$, and let
$$
A_0\in\Psibr^{2\tilde
  r-1,2l+1,2\nu+2,-\infty}(X)=\Psiscbr^{2\tilde r-1,2(\tilde
  r+l),2l+1,2l+1,-\infty}(X),
$$
have principal symbol
$$
x^{-2l-1}(\taub^2+\mub^2)^{\tilde
  r-1/2}.
$$
Then the principal symbol of
\begin{equation*}\begin{aligned}
&i[\re \hat
P(\sigma),A_0]+A_0S_0\hat P(\sigma)+\hat P(\sigma)^*S_0A_0
\end{aligned}\end{equation*}
in
$
\Psiscbr^{2\tilde r,2(\tilde
  r+l)-1,2l,2l,-\infty}(X)$
at $x=0$ is
\begin{equation}\begin{aligned}\label{eq:commutator-sc-version-imag-8}
&x^{-2l}(\taub^2+\mub^2)^{\tilde
  r-3/2}4\big(-(l+\tilde r) \taub^2+(l+1/2)\mub^2\big).
\end{aligned}\end{equation}
\end{lemma}

\begin{proof}
Adding to \eqref{eq:commutator-sc-version-imag}
$$
4(l+\tilde r)\re\tilde p(\sigma) x^{-2l-1}\taub(\taub^2+\mub^2)^{\tilde
  r-3/2}
$$
we get
\begin{equation}\begin{aligned}\label{eq:commutator-sc-version-imag-8b}
&x^{-2l}(\taub^2+\mub^2)^{\tilde
  r-3/2}\Big(4\big((l+\tilde r) \taub^2+(l+1/2)\mub^2\big)-8\taub^2(l+\tilde
r)\Big)\\
&=x^{-2l}(\taub^2+\mub^2)^{\tilde
  r-3/2}4\big(-(l+\tilde r) \taub^2+(l+1/2)\mub^2\big),
\end{aligned}\end{equation}
proving the lemma.
\end{proof}

Before proceeding, we state a more precise structure result for
$\im\tilde P(\sigma)$ which is crucial as the contribution of
$\im\tilde P(\sigma)$ to the operator we compute is higher order than
the commutator itself. This result states that $\im\tilde P(\sigma)$
has the same sign as $-(\im\sigma)$, in the operator theoretic sense,
modulo terms we can otherwise control or are irrelevant.

\begin{lemma}\label{lemma:im-P-sigma-structure}
We have
\begin{equation*}\begin{aligned}
&\im\tilde P(\sigma)=-(\im\sigma)T(\sigma)+W(\sigma) \in\Psiscbr^{2,0,-1,-1,1}(X),\\
&\qquad
T(\sigma)\in\Psiscbr^{2,0,-1,-1,2}(X),\qquad W(\sigma)\in\Psiscbr^{1,-1,-1,-1,1}(X)
\end{aligned}\end{equation*}
with
$$
T(\sigma)=T=\sum_j
T_j^2+\sum_j T_j T'_j+\sum_j T'_j T_j+T''_j
$$
with $T_j=T_j^*\in\Psiscbr^{1,0,-1,0,1}(X)$ (where $T_j$ is $|\sigma|^{-1}$
times the $T_j$ of \eqref{eq:hat-P-0-nonnegative}),
$T'_j=(T'_j)^*\in\Psiscbr^{0,-1,-1,0,1}(X)$,
$T''_j=(T''_j)^*\in\Psiscbr^{0,-2,-2,0,2}(X)$, $W=W^*$, so $T'_j,W$ are one order lower
than $T$ in terms of sc-decay, $T''_j$ two orders lower, and where
$W(\sigma)$ has principal symbol
$$
\im \Big( x^2\taub\beta\frac{\overline{\sigma}}{|\sigma|^2}-x(\beta-\gamma)+x\sigma\varpi\Big).
$$
\end{lemma}

\begin{rem}
While we could factor out $\im\sigma$ from the $\im(x\sigma\varpi)$
term in $W$ by our
assumptions, it is only $O(x)$, so could not be absorbed in $T''$.
\end{rem}

\begin{proof}
This is an immediate consequence of \eqref{eq:full-hat-P-sigma-separated},
\eqref{eq:hat-P-0-nonnegative} and \eqref{eq:hat-Q-module-form}.
\end{proof}

\begin{prop}
Suppose $\beta$ is skew-symmetric.
There exists $\sigma_0>0$ such that for $|\im\sigma|<\sigma_0$
we have
\begin{equation*}\begin{aligned}
&\|(1+x/|\sigma|)^{\alpha}u\|_{\Hscbr^{s,\tilde r+l-1/2,l}}\\
&\leq C(\|(1+x/|\sigma|)^{\alpha}(x+|\sigma|)^{-1}\hat
P(\sigma)u\|_{\Hscbr^{s-2,\tilde r+l+1/2,l+1}}+\|(1+x/|\sigma|)^{\alpha}u\|_{\Hscbr^{-N,-N,l}}),
\end{aligned}\end{equation*}
provided that
\begin{equation}\label{eq:symbolic-est-imag-constr}
0\leq \im\sigma,\qquad l+\tilde
r-\frac{\im(\beta+\gamma)}{2}>0,\qquad
l+1/2-\frac{\im(\beta-\gamma)}{2}<0,
\end{equation}
or if all inequalities in the constraints are
reversed.

This estimate holds in the same sense as Proposition~\ref{prop:symbolic-est-real}.
\end{prop}

\begin{proof}
Let
$$
A\in \Psibr^{2\tilde
  r-1,2l+1,-2\tilde\nu+2l+1,-\infty}(X)=\Psiscbr^{2\tilde r-1,2(\tilde
  r+l),2l+1,-2\tilde\nu+2l+1,-\infty}(X)
$$
with principal symbol
$$
a=x^{-2l-1}|\sigma|^{2\tilde\nu}(\taub^2+\mub^2)^{\tilde
  r-1/2}\phi(x/|\sigma|)^2,
$$
with $\phi\geq 0$ as above. 
Notice that \eqref{eq:commutator-sc-version-imag-8}
is $\leq 0$ if $l+1/2<0$, $l+\tilde r>0$ and $\geq 0$ if
$l+1/2>0$, $l+\tilde r<0$, matching the sign of the principal symbol
of $\im \tilde
P(\sigma) A+A\tilde P(\sigma)$ term if $\im\sigma\geq 0$, resp.\
$\im\sigma\leq 0$. Note also that the subprincipal terms of
$\tilde P(\sigma)$ arising from $\beta,\gamma$ and $\varpi$, namely the $W$
term in Lemma~\ref{lemma:im-P-sigma-structure}, yield a
contribution to $\im \tilde
P(\sigma) A+A\im\tilde P(\sigma)$ in $\Psiscbr^{2\tilde r,2(\tilde
  r+l)-1,2l,-2\tilde\nu+2l,-\infty}(X)$ which has principal symbol
\begin{equation*}\begin{aligned}
&2\im \Big(
x^2\taub\beta\frac{\overline{\sigma}}{|\sigma|^2}-x(\beta-\gamma)+x\sigma\varpi\Big)a\\
&=2x^2\taub(\im\beta)\frac{\re\sigma}{|\sigma|^2}a-2x^2\taub\re\beta\frac{\im\sigma}{|\sigma|^2}a-2x\im(\beta-\gamma)a+2x(\im\sigma)\varpi
a.
\end{aligned}\end{equation*}
Assuming that $\re\beta=0$, adding to this
\begin{equation*}\begin{aligned}
&-2(\im\beta)(\re\tilde p(\sigma)) \taub(\taub^2+\mub^2)^{-1}a\\
&=-2(\im\beta)(\re\tilde p(\sigma)) x^{-2l-1}\taub(\taub^2+\mub^2)^{\tilde
  r-3/2}\sigma^{2\tilde\nu}\phi(x/|\sigma|)
\end{aligned}\end{equation*}
we obtain
\begin{equation*}\begin{aligned}
&4x(\im\beta)(\taub^2+\mub^2)^{-1}\taub^2a-2x\im(\beta-\gamma)a+2x(\im\sigma)\varpi
a\\
&=2x\Big((\im(\beta+\gamma)+(\im\sigma)\varpi)\taub^2-(\im(\beta-\gamma)-(\im\sigma)\varpi)\mub^2\Big)(\taub^2+\mub^2)^{-1}a.
\end{aligned}\end{equation*}
Let
$$
\hat s=\hat s_0-(\im\beta)\taub(\taub^2+\mub^2)^{-1}
$$
be the principal symbol of
$S\in\Psiscbr^{-1,-1,0,0,0}(X)=\Psibr^{-1,0,0,0}(X)$. We can arrange that
$$
A=A_1^2,\ A_1=A_1^*,\qquad A_1\in \Psiscb^{\tilde r-1/2,\tilde
  r+l,l+1/2,-\tilde\nu+l+1/2,-\infty}(X)
$$
by choosing $A_1$ first with
the desired principal symbol.
Then
\begin{equation*}\begin{aligned}
TA+AT&=TA_1^2+A_1^2 T=2A_1 TA_1+[T,A_1]A_1+A_1[A_1,T]\\
&=2A_1TA_1+[[T,A_1],A_1],
\end{aligned}\end{equation*}
and now the second term is two orders lower than first due to the
double commutator.
Combined with \eqref{eq:commutator-sc-version-imag-8} this gives that
\begin{equation}\begin{aligned}\label{eq:commutator-sc-version-imag-88}
&i(\hat P(\sigma)^*A-A\hat P(\sigma)) +AS\hat P(\sigma)+\hat
P(\sigma)^*SA\\
&=(\im \hat P(\sigma) A+A\im \hat P(\sigma))+i[\re \hat P(\sigma),A]+AS\hat P(\sigma)+\hat P(\sigma)^*SA
\end{aligned}\end{equation}
in
$
\Psiscbr^{2\tilde r+1,2(\tilde
  r+l),2l,2l-2\tilde\nu,-\infty}(X)$
is
of the form
\begin{equation}\begin{aligned}\label{eq:commutator-sc-version-imag-lead}
2(\im\sigma)A_1TA_1=2(\im\sigma)\sum_j A_1^*T_j^*T_j
A_1&+2(\im\sigma)A_1^*T_j^*T'_j A_1\\
&+2(\im\sigma)A_1^*(T'_j)^*T_j A_1,
\end{aligned}\end{equation}
plus a term in $\Psiscbr^{2\tilde r-1,2(\tilde
  r+l)-1,2l,2l-2\tilde\nu,-\infty}(X)$ whose principal symbol is
\begin{equation}\begin{aligned}\label{eq:commutator-sc-version-imag-16}
\sigma^{2\tilde\nu}x^{-2l}(\taub^2+\mub^2)^{\tilde
  r-3/2}4\Big(&-\Big(l+\tilde
r-\frac{\im(\beta+\gamma)-(\im\sigma)\varpi}{2}\Big) \taub^2\\
&\qquad+\Big(l+1/2-\frac{\im(\beta-\gamma)-(\im\sigma)\varpi}{2}\Big)\mub^2\Big)
\end{aligned}\end{equation}
times $\phi(x/\sigma)^2$, plus a term localized where we have elliptic estimates,
arising from $\phi'$.
For suitable small $\im\sigma$, the $(\im\sigma)\varpi$ terms can be
absorbed in the definite
$$
l+\tilde
r-\frac{\im(\beta+\gamma)}{2} ,\qquad
l+1/2-\frac{\im(\beta-\gamma)}{2}
$$
terms, so \eqref{eq:commutator-sc-version-imag-16} is of the form
$b^2$ with
\begin{equation*}\begin{aligned}
b=\sigma^{\tilde\nu}x^{-l}(\taub^2+\mub^2)^{\tilde
  r/2-3/4}2\Big(&-\Big(l+\tilde
r-\frac{\im(\beta+\gamma)-(\im\sigma)\varpi}{2}\Big) \taub^2\\
&\qquad+\Big(l+1/2-\frac{\im(\beta-\gamma)-(\im\sigma)\varpi}{2}\Big)\mub^2\Big)^{1/2}\phi(x/\sigma).
\end{aligned}\end{equation*}
Thus, \eqref{eq:commutator-sc-version-imag-16}
can be written as $B^*B+F$ with the principal symbol of symbol of $B$
being $b$,
\begin{equation*}\begin{aligned}
&B\in\Psiscbr^{\tilde r-1/2,\tilde
  r+l-1/2,l,l-\tilde\nu-1/2,-\infty}(X)=\Psibr^{\tilde
  r-1/2,l,l-\tilde\nu-1/2,-\infty}(X),\\
& F\in\Psiscbr^{2\tilde r, 2\tilde r+2l-2,2l,2l-2\tilde\nu-1,-\infty}(X).
\end{aligned}\end{equation*}

Now, applying both sides of
\eqref{eq:commutator-sc-version-imag-88} to $u$ and pairing with $u$
we have
\begin{equation}\begin{aligned}
&2(\im\sigma)\sum_j \|T_j
A_1 u\|^2+4(\im\sigma)\re\langle T_j A_1 u,T'_j
A_1u\rangle+\|Bu\|^2+\langle Fu,u\rangle\\
&=2\im\langle \hat
P(\sigma)u,Au\rangle+2\re\langle \hat P(\sigma)u,SAu\rangle.
\end{aligned}\end{equation}
The terms from the `cross terms' in \eqref{eq:commutator-sc-version-imag-lead} can be
estimated via Cauchy-Schwartz:
$$
|(\im\sigma)\langle  T_j A_1 u, T'_j A_1 u\rangle|\leq \im\sigma
(\ep\|T_jA_1u\|^2+\ep^{-1}\|T_j'A_1 u\|^2),
$$
and now, for sufficiently small $\ep>0$, the $T_j$ term can be absorbed into $\|T_jA_1u\|^2$ arising
from the first term of the right hand side of \eqref{eq:commutator-sc-version-imag-lead},
while the second term, with $|\im\sigma|^{1/2}$ included, has
$$
|\im\sigma|^{1/2} T'_j A_1\in \Psiscbr^{\tilde r-1/2,\tilde
  r+l-1,l-1/2,-\tilde\nu+l,-\infty}(X)=\Psibr^{\tilde r-1/2,l-1/2,-\tilde\nu+l,-\infty}(X)
$$
which, for $\phi$ with sufficiently small support, can be absorbed into $\|Bu\|^2$ as $B\in\Psibr^{\tilde r-1/2,l,l-\tilde\nu,-\infty}(X)$, so the orders are the same, but
there is extra vanishing in $x/\sigma$ in the strong b-decay
sense. Finally, as $b$ is an elliptic multiple of $x^{1/2}a_1$, we can
estimate
$|\langle \hat
P(\sigma)u,Au\rangle|$, modulo terms like the one given by $F$, by
$$
\|x^{-1/2} A_1 \hat P(\sigma)u\|\|B u\|\leq \ep^{-1}\|x^{-1/2} A_1 \hat P(\sigma)u\|^2+\ep\|B u\|^2,
$$
and now for sufficiently small $\ep$, the last term on the right hand
side can be absorbed in $\|Bu\|^2$. This, combined with the existing elliptic estimates,  gives the desired estimate
with $-N$ replaced by a half order improvement, but an iterative
argument gives the full conclusion.

The case of reversed inequalities is similar.
\end{proof}

\section{Normal operator}\label{sec:normal}
We now turn to normal operators.

We recall that $(\nu_-,\nu_+)$ is the central weight interval for
the scattering end, see \eqref{eq:central-sc-interval-1}, and in
particular for $\beta=0,\beta'=0$ we have
$$
(\nu_-,\nu_+)=\Big(-1-\frac{n-2}{2},-1+\frac{n-2}{2}\Big).
$$
As explained in the paragraph following
\eqref{eq:central-sc-interval-1}, a dilation invariant operator can be
considered from the perspective of either `end' of the dilation orbit,
which concretely for the rescaled variable $X=x/\sigma$ means either
from the scattering end, where $X\to 0$, or the conic point end, where
$X^{-1}\to 0$, and as the Mellin transform exponent changes sign under
this perspective change, the weights, including the central weight
interval, also change sign.
In part of this section the conic point perspective plays a bigger
role, so we also
use the corresponding indicial roots; we
write
$$
(\nuc_-,\nuc_+)=(-\nu_+,-\nu_-)
$$
for the corresponding interval, so for $\beta=0$, $\beta'=0$, we have
$$
(\nuc_-,\nuc_+)=\Big(1-\frac{n-2}{2},1+\frac{n-2}{2}\Big),
$$
and in general
\begin{equation}\begin{aligned}\label{eq:central-conic-interval-1}
(\nuc_-,\nuc_+)=
\Big(1-&\frac{\im\beta}{2}-\re\sqrt{-\frac{\beta^2}{4}+\Big(\frac{n-2}{2}\Big)^2+\beta'},
\\
&1-\frac{\im\beta}{2}+\re\sqrt{-\frac{\beta^2}{4}+\Big(\frac{n-2}{2}\Big)^2+\beta'}\,\Big),
\end{aligned}\end{equation}
where we recall that $\beta$ is skew-symmetric and $\re\beta'>\frac{\beta^2}{4}-\Big(\frac{n-2}{2}\Big)^2$.

The key point is to show that

\begin{prop}\label{prop:zero-normal}
Let $0<x_0'<x_0$. Suppose that $\tilde r,l\in\RR$,
$$
\alpha\in \Big(l-\nu_+,l-\nu_-\Big)=\Big(l+\nuc_-,l+\nuc_+\Big),
$$
and
\begin{equation}\begin{aligned}\label{eq:zero-normal-tilde-r-l-constraint}
&\text{either}\ \tilde r+l>-1/2+\frac{\im(\beta+\gamma)}{2},\ l<-1/2+\frac{\im(\beta-\gamma)}{2},\\
&\text{or}\ \tilde r+l<-1/2+\frac{\im(\beta+\gamma)}{2},\ l>-1/2+\frac{\im(\beta-\gamma)}{2}.
\end{aligned}\end{equation}
There is $C>0$ such that for distributions $v$ supported in $x\leq x'_0$,
$v\in\Hb^{\tilde r,l}$, 
\begin{equation}\label{eq:zero-normal-to-show}
\|(1+x/|\sigma|)^\alpha v\|_{\Hb^{\tilde r,l}}\leq
  C\|(1+x/|\sigma|)^\alpha (x+|\sigma|)^{-1}N_0(\hat
  P(\sigma))v\|_{\Hb^{\tilde r-1,l+1}},
\end{equation}
with $C$ independent of $\sigma$ with $\im\sigma\geq 0$, $\sigma\neq 0$. The same is true
if we take $\tilde r_0\in\RR$, $\tilde r_0<\tilde r$, satisfying the
same inequalities as $\tilde r$, and only assume that $v\in\Hb^{\tilde
  r_0,l}$.

Similarly, with $\tilde r+l$ replaced by $r$ in
\eqref{eq:zero-normal-tilde-r-l-constraint},
there is $C>0$ such that for distributions $v$ supported in $x\leq x'_0$,
$v\in\Hscbr^{s, r,l}$, 
\begin{equation}\label{eq:zero-normal-to-show2}
\|(1+x/|\sigma|)^\alpha v\|_{\Hscbr^{s,r,l}}\leq
  C\|(1+x/|\sigma|)^\alpha (x+|\sigma|)^{-1}N_0(\hat
  P(\sigma))v\|_{\Hscbr^{s-2, r+1,l+1}},
\end{equation}
with $C$ independent of $\sigma$ with $\im\sigma\geq 0$, $\sigma\neq 0$, and the
conclusion also holds if we take $s_0,r_0\in\RR$, $s_0<s$, $r_0<r$, satisfying the
same inequalities as $r$, and only assume that $v\in\Hscbr^{s_0,
  r_0,l}$.
\end{prop}

\begin{rem} Notice
that $\alpha=0$ is acceptable for suitable $l<-1/2$ for $n\geq 3$, while
for suitable $l>-1/2$ only if $n\geq 4$.
\end{rem}

We postpone the proof of this proposition to the end of this section,
and rather start by deducing its consequences.
The first is:

\begin{prop}\label{prop:b-res-imp-est}
For $\tilde r,l,\alpha$ as in Proposition~\ref{prop:zero-normal},
$\tilde r'$ arbitrary, $\delta>0$ sufficiently small, there is $C>0$
such that
we have
\begin{equation}\begin{aligned}\label{eq:b-res-imp-est}
&\|(x+|\sigma|)^{\alpha}u\|_{\Hscbr^{s,\tilde r+l-1/2,l}}\\
&\leq C(\|(x+|\sigma|)^{\alpha-1}\hat
P(\sigma)u\|_{\Hscbr^{s-2,\tilde r+l+1/2,l+1}}
+\|(x+|\sigma|)^\alpha
  x^\delta u\|_{\Hb^{\tilde r'+1,l}}).
\end{aligned}\end{equation}
\end{prop}

\begin{proof}[Proof of Proposition~\ref{prop:b-res-imp-est}.]
For the sake of notational convenience, we assume that $\sigma>0$ so
that we can avoid writing $|\sigma|$; the general case only needs
notational changes.

We first show the proposition under the assumption that $s\geq\tilde r-1/2$.
Due to Proposition~\ref{prop:zero-normal}, with $v$ having the support
properties indicated there, and as
$$
\hat P(\sigma)-N_0(\hat P(\sigma))\in x(x+\sigma)x^\delta\Diffb^2(X)+x(x+\sigma)\sigma\Diffb^2(X),
$$
we have
\begin{equation*}\begin{aligned}
&\|(1+x/\sigma)^\alpha v\|_{\Hb^{\tilde r,l}}\\
&\leq
  C(\|(1+x/\sigma)^\alpha (x+\sigma)^{-1}\hat
  P(\sigma)v\|_{\Hb^{\tilde r-1,l+1}}\\
&\qquad\qquad+\|(1+x/\sigma)^\alpha
  x^\delta v\|_{\Hb^{\tilde r+1,l}}+\|(1+x/\sigma)^\alpha
  \sigma v\|_{\Hb^{\tilde r+1,l}}).
\end{aligned}\end{equation*}
Furthermore, with $v=\chi(x)u$, where $\chi$ is supported near $0$, in
$x\leq x'_0<x_0$,
and is identically $1$ near $0$,
$$
\hat P(\sigma)v=\chi\hat P(\sigma)u+[\hat P(\sigma),\chi]u,
$$
with the second term supported in $x>0$ and $(1+x/\sigma)^\alpha
(x+\sigma)^{-1}$ times it controlled in $\Hb^{\tilde r-1,l+1}$ by
$\|(1+x/\sigma)^\alpha u\|_{\Hb^{\tilde r,l'}}$, $l'$ arbitrary, so
\begin{equation}\begin{aligned}\label{eq:b-res-imp-est-16}
&\|(1+x/\sigma)^\alpha \chi u\|_{\Hb^{\tilde r,l}}\\
&\leq C(\|(1+x/\sigma)^\alpha (x+\sigma)^{-1}\hat
  P(\sigma) u\|_{\Hb^{\tilde r-1,l+1}}\\
&\qquad\qquad+\|(1+x/\sigma)^\alpha
  x^\delta u\|_{\Hb^{\tilde r+1,l}}+\|(1+x/\sigma)^\alpha
  \sigma u\|_{\Hb^{\tilde r+1,l}}).
\end{aligned}\end{equation}

Thus, starting with the symbolic estimate
\begin{equation*}\begin{aligned}
&\|(1+x/\sigma)^{\alpha}u\|_{\Hscbr^{s,\tilde r+l-1/2,l}}\\
&\leq C(\|(1+x/\sigma)^{\alpha}(x+\sigma)^{-1}\hat
P(\sigma)u\|_{\Hscbr^{s-2,\tilde r+l+1/2,l+1}}+\|(1+x/\sigma)^{\alpha}u\|_{\Hscbr^{-N,-N,l}}),
\end{aligned}\end{equation*}
and estimating the last term by the above normal operator estimate,
with $\tilde r$ replaced by any $\tilde r'\in [-N,\tilde r-3/2)$, we
get, with
$s\geq\tilde r-1/2$ (so that the norm on the left hand side
is stronger than that of $\Hb^{\tilde r-1/2,l}$)
\begin{equation*}\begin{aligned}
&\|(1+x/\sigma)^{\alpha}u\|_{\Hscbr^{s,\tilde r+l-1/2,l}}\\
&\leq C(\|(1+x/\sigma)^{\alpha}(x+\sigma)^{-1}\hat
P(\sigma)u\|_{\Hscbr^{s-2,\tilde r+l+1/2,l+1}}\\
&\qquad\qquad+\|(1+x/\sigma)^\alpha
  x^\delta u\|_{\Hb^{\tilde r'+1,l}}+\|(1+x/\sigma)^\alpha
  \sigma u\|_{\Hb^{\tilde r'+1,l}}).
\end{aligned}\end{equation*}
Now for $\tilde r'$ small relative to $\tilde r$, the third term on
the right can
be absorbed into the left hand side, while the second is relatively
compact, so this is an estimate modulo compact errors.
Note that as $\sigma$ powers can be pulled out of the norms, this is
equivalent to
\begin{equation}\begin{aligned}\label{eq:b-res-imp-est-restate}
&\|(x+\sigma)^{\alpha}u\|_{\Hscbr^{s,\tilde r+l-1/2,l}}\\
&\leq C(\|(x+\sigma)^{\alpha-1}\hat
P(\sigma)u\|_{\Hscbr^{s-2,\tilde r+l+1/2,l+1}}
+\|(x+\sigma)^\alpha
  x^\delta u\|_{\Hb^{\tilde r'+1,l}}),
\end{aligned}\end{equation}
which completes the proof of the Proposition when $s\geq\tilde r-1/2$.

For the general case, we simply use the $\Hscbr$ version of
Proposition~\ref{prop:zero-normal}, which replaces
\eqref{eq:b-res-imp-est-16} with
\begin{equation*}\begin{aligned}
&\|(1+x/\sigma)^\alpha \chi u\|_{\Hscbr^{s,r,l}}\\
&\leq C(\|(1+x/\sigma)^\alpha (x+\sigma)^{-1}\hat
  P(\sigma) u\|_{\Hscbr^{s-2,r-1,l+1}}\\
&\qquad\qquad+\|(1+x/\sigma)^\alpha
  x^\delta u\|_{\Hscbr^{s, r+1,l}}+\|(1+x/\sigma)^\alpha
  \sigma u\|_{\Hscbr^{s, r+1,l}}).
\end{aligned}\end{equation*}

\end{proof}

We can then
finish the proof of the main theorem, as we do below, using a variant of the standard compactness considerations to
obtain that there is $\sigma_0>0$ such that
\begin{equation*}\begin{aligned}
&\|(1+x/\sigma)^{\alpha}u\|_{\Hscbr^{s,\tilde r+l-1/2,l}}\leq C\|(1+x/\sigma)^{\alpha}(x+\sigma)^{-1}\hat
P(\sigma)u\|_{\Hscbr^{s-2,\tilde r+l+1/2,l+1}},
\end{aligned}\end{equation*}
or equivalently
\begin{equation}\begin{aligned}\label{eq:b-res-est}
&\|(x+\sigma)^{\alpha}u\|_{\Hscbr^{s,\tilde r+l-1/2,l}}\leq C\|(x+\sigma)^{\alpha-1}\hat
P(\sigma)u\|_{\Hscbr^{s-2,\tilde r+l+1/2,l+1}},
\end{aligned}\end{equation}
hold for $0<|\sigma|\leq\sigma_0$ provided $\hat P(0)$ has trivial nullspace. Taking into account
that, by
\cite{Vasy:Limiting-absorption-lag}, for $\sigma\neq 0$, $\hat P(\sigma)$ is Fredholm of index zero on
the spaces stated in Theorem~\ref{thm:main} and
Theorem~\ref{thm:main-gen} (and is indeed shown to be even
invertible if $P(\sigma)=P(\sigma)^*$ for $\sigma$ real), and that in
that case the weight factors inside the norms are bounded above and below by positive
constants, hence can be dropped, the above
estimates, in which the subscript `$\res$', indicating the
zero energy behavior, can be dropped for fixed non-zero $\sigma$,
show
the invertibility of $\hat P(\sigma)$ for $\sigma$ with
$0<|\sigma|\leq\sigma_0$ by virtue of implying that the nullspace is trivial.
{\em This proves the
  main theorem, Theorem~\ref{thm:main}, as well as Theorem~\ref{thm:main-gen}!}

\begin{proof}[Proof of Theorem~\ref{thm:main} and of
  Theorem~\ref{thm:main-gen}]
It only remains to prove \eqref{eq:b-res-est}.
We proceed under the assumptions of, and using the notation of,
Proposition~\ref{prop:b-res-imp-est}, taking $\tilde r'$ so that $s>\tilde r'+1$, $\tilde r>\tilde r'+3/2$.
If \eqref{eq:b-res-est} is not true, there are sequences
$\sigma_j\to 0$ and $u_j$, for which one may
assume that $(x+\sigma_j)^{\alpha} u_j$ has unit norm in $\Hscbr^{s,\tilde r+l-1/2,l}$, and with
$(x+\sigma_j)^{\alpha-1} \hat P(\sigma_j)u_j\in \Hscbr^{s-2,\tilde r+l+1/2,l+1}$,
and such that $(x+\sigma_j)^{\alpha-1} \hat P(\sigma_j)u_j\to 0$ in $\Hscbr^{s-2,\tilde r+l+1/2,l+1}$. By
taking a subsequence (not shown in notation), using the sequential
compactness of the unit ball in $\Hb^{\tilde r'+1+\delta',l}$ in the weak
topology, and the continuity of the (family of) inclusion(s)
$\Hscbr^{s,\tilde r+l-1/2,l}\to \Hb^{\tilde
  r'+1+\delta',l}$, as well as the compactness of the
inclusion $\Hb^{\tilde r'+1+\delta',l}\to \Hb^{\tilde
  r'+1,l-\delta}$ for $\delta'>0$ sufficiently small, one may assume that
there is $v\in \Hb^{\tilde
  r'+1+\delta',l}$ such that $(x+\sigma_j)^{\alpha} u_j\to v$ weakly in
$\Hb^{\tilde
  r'+1+\delta',l}$ and strongly in $\Hb^{\tilde
  r'+1,l-\delta}$. By \eqref{eq:b-res-imp-est}
then $\liminf \|(x+\sigma_j)^{\alpha} u_j\|_{\Hb^{\tilde
    r',l-\delta}}\geq C^{-1}>0$, so $v\neq 0$ by the strong convergence. On the other hand,
$$
(x+\sigma_j)^{\alpha-1} \hat P(\sigma_j)u_j\to x^{\alpha-1}\hat P(0) x^{-\alpha}v
$$
in $\Hb^{\tilde r'-1,l-\delta}$ as
\begin{equation*}\begin{aligned}
&(x+\sigma_j)^{\alpha-1}\hat P(\sigma_j)u_j-x^{\alpha-1}\hat P(0) x^{-\alpha}v\\
&=(x+\sigma_j)^{\alpha-1} (\hat P(\sigma_j)-\hat
P(0)) (x+\sigma_j)^{-\alpha}  ((x+\sigma_j)^{\alpha}  u_j)\\
&\qquad\qquad+(x+\sigma_j)^{\alpha-1}\hat
P(0)(x+\sigma_j)^{-\alpha}((x+\sigma_j)^{\alpha}u_j-v)\\
&\qquad\qquad+(x+\sigma_j)^{\alpha-1}\hat
P(0)((x+\sigma_j)^{-\alpha}-x^{-\alpha})v\\
&\qquad\qquad+((x+\sigma_j)^{\alpha-1}-x^{\alpha-1})\hat P(0)x^{-\alpha}v
\end{aligned}\end{equation*}
since $\hat P(\sigma_j)\to \hat P(0)$ as bounded operators in $\cL(\Hb^{\tilde
  r'+1,l-\delta-\alpha},\Hb^{\tilde r'-1,l+1-\delta-\alpha})$,
$(x+\sigma_j)^{-\alpha}$, resp.\ $(x+\sigma_j)^{\alpha-1}$ are uniformly
bounded between $\Hb$
spaces whose weight differs by $-\alpha$, resp.\ $\alpha-1$, and $(x+\sigma_j)^{-\alpha}u_j$ converges to $v$
(thus is bounded) in $\Hb^{\tilde
  r'+1,l-\delta}$, while finally
$(x+\sigma_j)^{\kappa}-x^{\kappa}=\sigma_j\int_0^1\kappa(x+t\sigma_j)^{\kappa-1}\,dt$
shows that this difference goes to $0$ (due to the $\sigma_j$ factor)
as an operator being weighted spaces whose order differs by
$\kappa-1$. Thus, $\hat P(0)x^{-\alpha} v=0$, so $u=x^{-\alpha} v$ is a non-trivial element of the
nullspace of $\hat P(0)$ on $\Hb^{\tilde r-1/2,l-\alpha}$, with $l-\alpha\in(\nuc_-,\nuc_+)=(-1-\frac{n-2}{2},-1+\frac{n-2}{2})$, which contradicts our
assumptions. This proves \eqref{eq:b-res-est}, as desired.
\end{proof}

We now return to the proof of Proposition~\ref{prop:zero-normal}.

The main claim is that \eqref{eq:zero-normal-to-show} follows
from the basic scattering estimate, namely the limiting absorption
principle on corresponding spaces at energy $1$, applied on a
scattering manifold which also has a conic point in its interior;
we now recall this. 

We
are thus working with b-Sobolev spaces at both ends, denoted by
$\Hb^{\tilde r,l,\nu}$, $\nu$ the weight at the conic point, using the
scattering density at the sc-end, the conic density at the conic point, for the $L^2$-space (with trivial weight, i.e.\
$\nu=l=0$), or the second microlocal sc-b space at the scattering
end, but the standard b-space at the conic point, denoting it by
$\Hscb^{s,r,l,\nu}$, with $\nu$ the weight at the conic point. We
recall that for conic points indicial roots are the poles of the
Mellin transformed normal operator family in the b-pseudodifferential
algebra (normalized to make the family dilation invariant), and we
denote the `conic point central weight interval', see the beginning of
the section, by $(\nuc_-,\nuc_+)$
which
for
the actual Laplacian is
$$
(\nuc_-,\nuc_+)=\Big(1-\frac{n-2}{2},1+\frac{n-2}{2}\Big).
$$
The latter corresponds to the indicial
roots at the conic point. The conic density is $s^{n-1}\,ds\,dy$ near
$s=0$, so $s^{-n/2}$ is barely not in the weight $0$ space, hence the
allowed exponents are between those of $s^{2-n}$ and $s^0$.
We recall the
limiting absorption principle in the present,
conjugated setting:

\begin{prop}[Limiting absorption principle, \cite{Vasy:Limiting-absorption-lag}]
For $\tilde P(\varsigma)$ of conjugated spectral family type, with
a scattering and a conic end, $(\nuc_-,\nuc_+)$ the central weight interval for the conic point, and
the unconjugated operator being formally self-adjoint, and for
$\varsigma\neq 0$, $\im\varsigma\geq 0$, we have
$$
\|v\|_{\Hscb^{s,r,l,\nu}}\leq C\|\tilde P(\varsigma) v\|_{\Hscb^{s-2,r+1,l+1,\nu-2}},
$$
and
its slightly lossy (in terms of differentiability, in other aspects
its equivalent to the above version) version
\begin{equation}\label{eq:limabs-b-est}
\|v\|_{\Hb^{\tilde r,l,\nu}}\leq C\|\tilde P(\varsigma) v\|_{\Hb^{\tilde
    r,l+1,\nu-2}}.
\end{equation}
These estimates are valid if, first of all, $\nu\in
(\nuc_-,\nuc_+)$ and, for the sc-b case either $r>-1/2$ and $l<-1/2$, or $r<-1/2$
and $l>-1/2$, while for the b case either $\tilde r+l>-1/2$ and
$l<-1/2$, or $\tilde r+l<-1/2$ and $l>-1/2$.

If we drop the assumption of the unconjugated operator being formally
self-adjoint, these are replaced by Fredholm estimates, with $r>-1/2$
replaced by $r>-1/2+\frac{\im(\beta+\gamma)}{2}$, $l<-1/2$ replaced by
$l<-1/2+\frac{\im(\beta-\gamma)}{2}$, $\tilde r+l>-1/2$ replaced by
$\tilde r+l>-1/2+\frac{\im(\beta+\gamma)}{2}$, and similarly with the
inequalities reversed. If the
skew-adjoint part is sufficiently small, the above invertibility estimates remain valid.
\end{prop}

\begin{rem}\label{rem:limiting-absorption}
In  \cite{Vasy:Limiting-absorption-lag} conic points were not
considered explicitly. However, since the operator is elliptic near
the conic point in the b-sense, near them one has standard b-estimates
provided that the weight is not the negative
of the imaginary part of an indicial root, see e.g.\
\cite{Melrose:Atiyah}, cf.\ also \cite{Vasy:Minicourse} where the
analytically similar cylindrical ends are discussed. Since \cite{Vasy:Limiting-absorption-lag} uses
elliptic estimates away from the scattering end, this is a minor
change that fits seamlessly into the framework. This gives Fredholm
estimates directly. However, for formally self-adjoint operators, thus
with $\beta=0$ (since it is skew-adjoint by hypothesis!), for the
central weight interval for real $\varsigma$ the boundary
pairing formula goes through showing absence of kernel and cokernel,
thus proving the stated version of the proposition. Indeed,
\eqref{eq:central-conic-interval-1} means that
\begin{equation}\begin{aligned}\label{eq:central-conic-interval-1p}
(\nuc_-,\nuc_+)=
\Big(1-\re\sqrt{\Big(\frac{n-2}{2}\Big)^2+\beta'},
1+\re\sqrt{\Big(\frac{n-2}{2}\Big)^2+\beta'}\,\Big),
\end{aligned}\end{equation}
which includes a neighborhood of $1$ as
$\re\beta'>-\Big(\frac{n-2}{2}\Big)^2$ by hypothesis, so elements of
the nullspace are in weighted spaces of more than first order
vanishing at the conic point (in terms of weighted b-spaces, relative
to the metric $L^2$-weight), which suffices for the integration by
parts arguments. Moreover, for $\im\varsigma>0$ the standard pairing
formula for $P(\varsigma)$ showing the absence of nullspace works as
well: note that the undoing the conjugation that gave us $\tilde P(\varsigma)$ means multiplication of
the elements of the kernel by $e^{i\varsigma/X}$ (with $X$ denoting
the conic variable), which are thus
exponentially decaying at the scattering end (and no change at the
conic point). As the family is uniformly Fredholm, in the sense
discussed in \cite{Vasy-Dyatlov:Microlocal-Kerr}, the index is
constant $0$ (as it is such at the real axis where we have
invertibility), so this also shows the absence of cokernel for $\im\varsigma>0$. Notice that it is at this point that the choice of
the central weight interval (as opposed to another weight interval
with negative
free of indicial roots) plays a key role. Since the estimates are
perturbation stable, the conclusions hold for sufficiently small
skew-adjoint parts.
\end{rem}

\begin{proof}[Proof of Proposition~\ref{prop:zero-normal}.]
We
apply the scattering-conic estimate for our model operator which takes the form
$$
N_0(\hat P(\sigma))=\Delta_0+\beta \Big(x^3D_x+i\frac{n-2}{2}x^2\Big)+\beta'x^2-2\sigma x\Big(xD_x+i\frac{n-1}{2}+\frac{\beta-\gamma}{2}\Big)
$$
where $\Delta_0$ is the Laplacian of the {\em exact} conic metric
$g_0=\frac{dx^2}{x^4}+\frac{h_0}{x^2}$, $h_0$ the metric on $\pa X$,
after suitable rescaling via $X=x/|\sigma|$. Under this rescaling,
with $\hat\sigma=\sigma/|\sigma|$,
$N_0(\hat P(\sigma))$ becomes $|\sigma|^2\tilde P(\hat\sigma)$, where the
rescaled operator is
$$
\tilde P(\varsigma)=\Delta_0+\beta \Big(X^3 D_X+i\frac{n-2}{2}X^2\Big)+\beta'X^2-2 \varsigma X\Big(X D_X+i\frac{n-1}{2}+\frac{\beta-\gamma}{2}\Big),
$$
with $\Delta_0$ still the Laplacian of the exact conic metric
$\frac{dX^2}{X^4}+\frac{h_0}{X^2}$, and $\tilde P(\varsigma)$ indeed
enjoys the above limiting absorption principle estimate.

Now, the unweighted b-Sobolev spaces {\em relative to a b-density} are dilation invariant since the
b-derivatives are such, and the b-density corresponds to weights
$-n/2$, resp.\ $n/2$ at the sc-end, resp.\ the conic point, relative to a scattering density in our normalization, so under the map $\kappa_\sigma:(x,y)\to
(x/\sigma,y)$,
$$
\|\kappa_\sigma^* v\|_{\Hb^{\tilde r,-n/2,n/2}}=\|v\|_{\Hb^{\tilde r,-n/2,n/2}}.
$$
(Note that there is no support condition on $v$ here!)
Since
$$
\|v\|_{\Hb^{\tilde r,l,\nu}}=\|(1+X)^{\nu-n/2} (1+X^{-1})^{l+n/2}
v\|_{\Hb^{\tilde r,-n/2,n/2}},
$$
we have
\begin{equation*}\begin{aligned}
\|v\|_{\Hb^{\tilde r,l,\nu}}&=\|(1+X)^{\nu-n/2} (1+X^{-1})^{l+n/2}
v\|_{\Hb^{\tilde r,-n/2,n/2}}\\
&=\|(1+x/\sigma)^{\nu-n/2} (1+\sigma/x)^{l+n/2}
\kappa_\sigma^*v\|_{\Hb^{\tilde r,-n/2,n/2}}\\
&=\|(1+x/\sigma)^{\nu+l}
x^{-l-n/2}\sigma^{l+n/2}\kappa_\sigma^*v\|_{\Hb^{\tilde r,-n/2,n/2}}\\
&=|\sigma|^{l+n/2} \|(1+x/\sigma)^{\nu+l}\kappa_\sigma^*v\|_{\Hb^{\tilde r,l,-l}}.
\end{aligned}\end{equation*}
Thus, \eqref{eq:limabs-b-est} becomes
\begin{equation*}\begin{aligned}
|\sigma|^{l+n/2}&
\|(1+x/\sigma)^{\nu+l}\kappa_\sigma^*v\|_{\Hb^{\tilde r,l,-l}}\\
&=\|v\|_{\Hb^{\tilde r,l,\nu}}\\
&\leq C\|\tilde P(1) v\|_{\Hb^{\tilde
    r,l+1,\nu-2}}\\
&=C|\sigma|^{l+1+n/2}\|(1+x/\sigma)^{\nu+l-1}\kappa_\sigma^*\tilde P(1) v\|_{\Hb^{\tilde r,l+1,-l-1}}\\
&=C|\sigma|^{-2}|\sigma|^{l+1+n/2}\|(1+x/\sigma)^{\nu+l-1} N_0(\hat P(\sigma)) \kappa_\sigma^*v\|_{\Hb^{\tilde 
    r,l+1,-l-1}}\\
&=C|\sigma|^{l+n/2}\|(1+x/\sigma)^{\nu+l}(x+\sigma)^{-1} N_0(\hat P(\sigma)) \kappa_\sigma^*v\|_{\Hb^{\tilde 
    r,l+1,-l-1}}, 
\end{aligned}\end{equation*}
which in summary gives
\begin{equation}\begin{aligned}\label{eq:rescaled-b-est}
\|(1+x/\sigma)^{\nu+l}\kappa_\sigma^*v\|_{\Hb^{\tilde r,l,-l}}&\leq C\|(1+x/\sigma)^{\nu+l} (x+\sigma)^{-1} N_0(\hat P(\sigma)) \kappa_\sigma^*v\|_{\Hb^{\tilde
    r,l+1,-l-1}},
\end{aligned}\end{equation}
which is exactly \eqref{eq:zero-normal-to-show}, once we replace
$\kappa_\sigma^*v$ by $v$, $\nu+l$ by $\alpha$ and keep in mind that
due to the support condition, only the
weight $l+1$ on the sc-end is relevant.

The computation for the sc-b-resolved spaces is completely similar:
the additional microlocal weights involve $\tau=x\taub$ and
$\mu=x\mub$ which means that the corresponding weights using
$\tausc=X\taub$ and $\musc=X\mub$ yield
$\|\kappa_\sigma^*
v\|_{\Hscbr^{s,r,-n/2,n/2}}=\|v\|_{\Hscb^{s,r,-n/2,n/2}}$, so
following the above computation gives
\begin{equation}\begin{aligned}\label{eq:rescaled-sc-b-est}
&\|(1+x/\sigma)^{\nu+l}\kappa_\sigma^*v\|_{\Hb^{s,r,l,-l}}\\
&\qquad\leq C\|(1+x/\sigma)^{\nu+l} (x+\sigma)^{-1} N_0(\hat P(\sigma)) \kappa_\sigma^*v\|_{\Hscbr^{s-2,
    r,l+1,-l-1}},
\end{aligned}\end{equation}
completing the proof.
\end{proof}

\section{Zero energy nullspace}\label{sec:zero}
We now discuss, assuming $\beta=0$, $\beta'=0$ and $\gamma=0$, what happens in the presence of non-trivial nullspace
of $\hat P(0)$, i.e.\ of $P(0)$, on the relevant function space,
$\Hb^{\infty,l-1/2}$. As discussed in \cite{Vasy:Zero-energy}, one issue is that the domain of $\hat P(\sigma)$ varies with $\sigma$ in a
serious way in that, depending on the dimension, $\Ker \hat P(0)$ need not
lie in the domain of $\hat P(\sigma)$; another issue is that the
resolution we introduced here is not very easy to use for perturbation
theory directly (since one needs a smooth, continuous, etc., {\em family} of operators). As in \cite{Vasy:Zero-energy},
and indeed following a long tradition in scattering theory,
this can be remedied by letting $\check P(\sigma)$ be a perturbation
of $\hat P(\sigma)$ in the same class but with $\check P(0)$
invertible, where one can arrange this with
$$
V(\sigma)=\check P(\sigma)-\hat
P(\sigma)
$$
even compactly supported in the interior of $X$, though
this is not necessary. Then one considers
$$
\hat P(\sigma)\check P(\sigma)^{-1}: \cY\to\cY,\ \cY=\Hb^{\tilde r-1,l+3/2};
$$
its invertibility at $\sigma=0$ is equivalent to that of $\hat
P(0)$ on the above discussed space. We have
$$
\hat P(\sigma)\check P(\sigma)^{-1}=\Id-V(\sigma)\check P(\sigma)^{-1}.
$$
The nullspace of this is the image of that of $\hat P(0)$ under
$\check P(0)$, while the $L^2$-orthocomplement of the range is the
nullspace of $P(0)^*$ in $\cY^*$.
One can decompose $\cY$ into $\check P(0) \Ker \hat P(0)$ and its
orthocomplement, and similarly on the target space side into $\Ran
\hat P(0)$ and $\Ker \hat P(0)^*$. Since $\hat P(0)u=0$ means $\check
P(0)u=V(0)u$, we have $\check P(0)|_{\Ker \hat P(0)}=V(0) |_{\Ker \hat
  P(0)}$. Thus, the entries of the block matrix from $\Ker \hat
P(0)\check P(0)^{-1}=\check P(0) \Ker\hat P(0)$ are
\begin{equation}\begin{aligned}\label{eq:Ker-comp-8}
\hat
P(\sigma)\check P(\sigma)^{-1}\check P(0)|_{\Ker \hat P(0)}&=\check
P(0)-V(\sigma)\check P(\sigma)^{-1}\check P(0)|_{\Ker \hat P(0)}\\
&=V(0)-V(\sigma)\check P(\sigma)^{-1}\check P(0)|_{\Ker \hat P(0)}\\
&=V(0)-V(\sigma)+V(\sigma) (\check P(0)^{-1}-\check P(\sigma)^{-1})V(0)|_{\Ker \hat P(0)}.
\end{aligned}\end{equation}
Similarly, for $\Ker \hat P(0)^*$, we have
\begin{equation*}\begin{aligned}
(\check P(\sigma)^{-1})^*\hat P(\sigma)^*|_{\Ker \hat
  P(0)^*}&=\Id-(\check P(\sigma)^{-1})^* V(\sigma)^*|_{\Ker \hat
  P(0)^*}\\
&=((\check P(0)^{-1})^*-(\check P(\sigma)^{-1})) V(\sigma)^*\\
&\qquad+(\check P(0)^{-1})^*(V(0)-V(\sigma)) |_{\Ker \hat
  P(0)^*}.
\end{aligned}\end{equation*}
Now, the argument would proceed by using that $V(0)-V(\sigma)$ and
$\check P(0)^{-1}-\check P(\sigma)^{-1}$ (and also the difference
of adjoints) are small as $\sigma\to 0$. This is clear in the case of
$V(0)-V(\sigma)$, as it is $O(\sigma)$ with suitably decaying
coefficients, but it is much less so for $(\check P(0)^{-1})-(\check
P(\sigma)^{-1})$.

Formally (thus imprecisely),
$$
\check P(0)^{-1}-\check P(\sigma)^{-1}=\check P(0)^{-1}(\check
P(\sigma)-\check P(0))\check P(\sigma)^{-1}=\check P(\sigma)^{-1}(\check
P(\sigma)-\check P(0))\check P(0)^{-1},
$$
but the composition on the right does not make sense as $\check
P(0)^{-1}$ loses 2 orders of decay, and $\check
P(\sigma)-\check P(0)$ is $O(\sigma)$ with a gain of one order of
decay, which is not sufficient in general for the result to be in the
domain of $P(\sigma)^{-1}$. {\em Notice that this is still better than the
unconjugated case, where $\check
P(\sigma)-\check P(0)$ is $O(\sigma^2)$ but gains no decay at all.} However, when applied to $V(0)|_{\Ker \hat
  P(0)}$, $\check P(0)^{-1}$ gives the element of $\Ker \hat
  P(0)$. Now, the great advantage of the present conjugated setting is
  that in fact $\check
P(\sigma)-\check P(0)|_{\Ker \hat
  P(0)}$ is $O(\sigma)$ with 2 orders of decay gained. This follows
from this difference, modulo terms with the claimed 2 orders of decay,
being
$$
-2\sigma x(xD_x+i(n-1)/2).
$$
Now, if $\lambda_j$ are the eigenvalues of $\Delta_{\pa X}$, then
elements of $\Ker\hat P(0)$ have an expansion starting with
$\frac{n-2}{2}+\sqrt{(\frac{n-2}{2})^2+\lambda_j}$. Thus, if the
$\lambda_j$ are the spherical eigenvalues (thus non-negative
integers), all terms, when an $O(x^2)$ b-differential operator is
applied to them, and all, except the $\lambda_0=0$
term in case $n=3$, when an $O(x)$ b-differential operator is applied to them, are
in the full range of decay orders for the domain of $\hat P(\sigma)$, as they are in
$x^2\Hb^{\infty,-1+\frac{n-2}{2}-\ep}=\Hb^{\infty,\frac{n}{2}-\ep}$, resp.\ $x\Hb^{\infty,-1+\frac{n-2}{2}+1-\ep}=\Hb^{\infty, \frac{n}{2}-\ep}$,  
and the acceptable range
of weights for the domain start (on the most decaying end) at $\frac{3}{2}-\ep$, so
$\frac{n}{2}\leq \frac{3}{2}$ shows the claim. For the remaining
$\lambda_0$ case with an $O(x)$ b-differential operator applied to it
the result is in $x\Hb^{\infty,-1+\frac{n-2}{2}-\ep}=\Hb^{\infty,
  \frac{n}{2}-1-\ep}$, which is in the acceptable range
of weights for the domain which starts, on the least decaying end, at
$-\frac{n-4}{2}+\ep$, so provided $-\frac{n-4}{2}<\frac{n}{2}-1$, which holds
for $n>3$, shows that for $n\geq 4$ in fact all terms make sense.
It remains to consider the most
singular case, $n=3$, but in fact even then this
holds for the $\lambda_0$ term since $(xD_x+i(n-1)/2)$ annihilates the
leading order asymptotics as $\frac{n-1}{2}=n-2$ in this case. Thus,
assuming our formal computation can be justified, all of the terms,
other than that of the big block ($00$ block) are $O(|\sigma|)$. Moreover,
if the the $11$ block is of the form $\sigma$ times an
invertible operator, Gaussian elimination shows that the whole
operator is invertible for $\sigma\neq 0$, and the inverse has the
form
\begin{equation}\label{eq:inv-block-bounds}
\begin{pmatrix} O(1)&O(1)\\O(1)&O(|\sigma|^{-1})\end{pmatrix}.
\end{equation}

Now, unlike in the unconjugated case, when for $\sigma\neq 0$ the
computation makes sense directly by wave front set considerations, we
need to explicitly justify it here even then. Thus, we need to consider
\begin{equation*}\begin{aligned}
\check P(0)^{-1}-\check P(\sigma)^{-1}&=\check P(0)^{-1}-\check
P(\sigma)^{-1}(\check P(0)\check P(0)^{-1})
\end{aligned}\end{equation*}
and reparenthesize the last term, which we do by inserting a
regularizer (in terms of decay) $\chi_\ep=\chi(./\ep)$, $\chi\geq 0$ smooth,
$\chi\equiv 1$ on $[1,\infty)$, supported in $(0,\infty)$.
Thus, with
the limits being strong operator limits,
\begin{equation*}\begin{aligned}
\check P(0)^{-1}-\check P(\sigma)^{-1}&=\lim_{\ep\to 0}(\check P(\sigma)^{-1} \check P(\sigma))\chi_\ep \check P(0)^{-1}-\check
P(\sigma)^{-1}\chi_\ep (\check P(0) \check P(0)^{-1})\\
&=\lim_{\ep\to 0}\Big((\check P(\sigma)^{-1} \check P(\sigma))\chi_\ep \check P(0)^{-1}-(\check
P(\sigma)^{-1}\chi_\ep \check P(0)) \check P(0)^{-1}\Big)\\
&=\lim_{\ep\to 0}\Big(\check P(\sigma)^{-1} \check P(\sigma)\chi_\ep -\check
P(\sigma)^{-1}\chi_\ep \check P(0) \Big) \check P(0)^{-1}\\
&=\lim_{\ep\to 0}\Big(\check P(\sigma)^{-1} \check P(\sigma) \chi_\ep+\check P(\sigma)^{-1} [\check P(0),\chi_\ep] -\check
P(\sigma)^{-1}\check P(0) \chi_\ep \Big) \check P(0)^{-1}\\
&=\lim_{\ep\to 0}\check P(\sigma)^{-1} \Big(\check P(\sigma) \chi_\ep+
[\check P(0),\chi_\ep] -\check P(0) \chi_\ep \Big) \check P(0)^{-1}\\
&=\check P(\sigma)^{-1} \lim_{\ep\to 0}\Big((\check P(\sigma) -\check
P(0)) \chi_\ep+ [\check P(0),\chi_\ep] \Big) \check P(0)^{-1}
\end{aligned}\end{equation*}
which on $\Ker\hat P(0)$, where we already saw that the first term indeed
makes sense by direct domain considerations, is
$$
\check P(\sigma)^{-1} (\check P(\sigma) -\check
P(0)) \check P(0)^{-1}
+\check P(\sigma)^{-1}
\lim_{\ep\to 0}[\check P(0),\chi_\ep] \check P(0)^{-1},
$$
with similar considerations applying to $\Ker \hat P(0)^*$. Now, the
commutator is uniformly (in $\ep$) in $\Psib^{2,-2}(X)$, and is tending
to $0$ in the strong topology
$$
\cL(\Hb^{\tilde r,l-1/2},\Hb^{\tilde
  r-2,l+3/2}),
$$
so this term vanishes in the limit.
This completes the justification
of our computation, and thus that
$$
\check P(0)^{-1}-\check P(\sigma)^{-1}
$$
is $O(|\sigma|)$. This also shows that modulo $O(|\sigma|^2)$, the
result is
$$
\check P(0)^{-1} (\check P(\sigma) -\check
P(0)) \check P(0)^{-1}.
$$

In summary then, mod $O(|\sigma|^2)$, \eqref{eq:Ker-comp-8} becomes
\begin{equation}\begin{aligned}\label{eq:Ker-comp-8b}
&V(0)-V(\sigma)+V(\sigma) (\check P(0)^{-1}-\check
P(\sigma)^{-1})V(0)|_{\Ker \hat P(0)}\\
&=V(0)-V(\sigma)+\check P(\sigma) -\check
P(0) |_{\Ker \hat P(0)}\\
&=\hat P(\sigma)-\hat P(0) |_{\Ker \hat P(0)}.
\end{aligned}\end{equation}
Thus, if this pairing between $\Ker P(0)$ and $\Ker P(0)^*$ is a
non-degenerate multiple of $\sigma$, we have that the inverse
satisfies bounds as in \eqref{eq:inv-block-bounds}.

\def\cprime{$'$} \def\cprime{$'$}

\end{document}